\documentclass[a4paper, leqno]{article}
\usepackage[english]{babel}
\usepackage{amsmath,amscd,amssymb,amsthm}
\usepackage{pb-diagram}
\usepackage{graphicx}
\usepackage{tikz-cd}
\usepackage{bm}
\usepackage{tikz} 
\usepackage{centernot,cancel}
\usepackage[left]{lineno}

\newtheorem{definition}{Definition}[section]
\newtheorem{theorem}{Theorem}[section]
\newtheorem{lemma}{Lemma}[section]
\newtheorem{proposition}{Proposition}[section]
\newtheorem{corollary}{Corollary}[section]

\newtheorem{remark}{Remark}[section]
\newtheorem{claim}{Claim}[section]

\begin{document}

\title{Moduli Spaces of the Basic Hitchin Equation on Sasakian Threefolds}

\author{Takashi Ono\thanks{Research Institute for Mathematical Sciences, Kyoto University, Kyoto, Japan, takashio@kurims.kyoto-u.ac.jp}}
\date{}
\maketitle

  \begin{abstract}
In this paper, we study an equation which we call the basic Hitchin equation. 
This is an equation defined on Sasakian threefolds and is a three-dimensional analog of the Hitchin equation, which is defined on Riemann surfaces.
We construct the moduli space of the basic Hitchin equation and show that such a space admits a hyperK\"ahler metric. This also shows that the moduli space of flat bundles over Sasakian threefolds admits a hyperK\"ahler metric. We also calculate the dimension of the moduli space.
\end{abstract}
\section{Introduction}
Let $X$ be a compact Riemann surface of genus at least two.
 Let $E$ be a complex vector bundle over $X$ and $h$ be a Hermitian metric. 
 Let $(\nabla_h,\Phi)$ be a pair of a $h$-unitary connection and an $\mathrm{End}E$-valued skew-symmetric 1-form w.r.t. $h$. 
 In \cite{H}, Hitchin considered the following equations:
\begin{align*}
F_{\nabla_h}-\Phi\wedge\Phi&=0,\\
\nabla_h\Phi&=0,\\
\nabla_h\ast\Phi&=0.
\end{align*}
Here $F_{\nabla_h}$ is the curvature of $\nabla_h$ and $\ast$ is the Hodge star. 
This equation is called the \textit{Hitchin equation.} 
We say $(\nabla_h,\Phi)$ is a Hitchin pair if it satisfies the Hitchin equation and is irreducible if the connection $D:=\nabla_h+\sqrt{-1}\Phi$ is irreducible.
In \cite{H}, he also constructed the moduli space $\mathcal{M}_{\mathrm{Hit}}$ of irreducible Hitchin pairs by infinite-dimensional hyperK\"ahler reduction and showed that it carries a natural hyperK\"ahler metric. \par
Hitchin equation is related to \textit{Higgs bundles} and \textit{flat bundles}.
To consider a Hitchin pair is equivalent to considering a polystable Higgs bundle of degree 0 and semisimple flat bundles.
An irreducible Hitchin pair corresponds to a stable Higgs bundle with degree 0 and a simple flat bundle. 
Hence, we can regard $\mathcal{M}_{\mathrm{Hit}}$ as a moduli space of stable Higgs bundles and simple flat bundles. 
$\mathcal{M}_{\mathrm{Hit}}$ intersects with many subjects, and the research of its properties is one of the active topics in modern mathematics.
\par
Let $M$ be a compact Sasakian manifold. Sasakian manifolds are odd-dimension analogs of K\"ahler manifolds.
See \cite{BG} for more details about Sasakian manifolds.
In this paper, we focus on the case of $\mathrm{dim}M=3$. 
We call such an $M$ a Sasakian threefold. In this case, $M$ is a three-dimensional analog of the Riemann surface.\par
On Sasakian manifolds, a natural analog of Higgs bundles is given by the notion of basic Higgs bundles. 
Basic Higgs bundles and their relation to flat bundles on compact Sasakian manifolds have been studied previously (see e.g. \cite{BH, BH2, WZ}).
Motivated by the classical Hitchin equation and the theory of basic Higgs bundles, 
we formulate the corresponding gauge-theoretic equations on Sasakian threefolds, which we call the \textit{basic Hitchin equation}.\par
Let $E$ be a \textit{basic} complex vector bundle and $h$ be a \textit{basic} Hermitian metric (See Section \ref{basic} for definitions about basic vector bundles and metrics). 
Let $(\nabla_h,\Phi)$ be a pair of basic $h$-unitary connection and $\Phi$ be a basic $\mathrm{End}E$-valued skew-symmetric 1-form w.r.t. $h$. 
Then the basic Hitchin equation is the following:
\begin{align*}
F_{\nabla_h}-\Phi\wedge\Phi&=0,\\
\nabla_h\Phi&=0,\\
\nabla_h\star_\xi\Phi&=0. 
\end{align*} 
Here $\star_\xi$ is the basic Hodge star (See Section \ref{b-pre}). We call a pair $(\nabla_h,\Phi)$ a basic Hitchin pair if the pair satisfies the basic Hitchin equation.
The main result of this paper is the construction of the moduli space $\mathcal{M}^{\mathrm{irr}}_{\mathrm{BaHit}}$ of irreducible basic Hitchin pairs. 
Moreover, we have
\begin{theorem}[Theorem \ref{main-th}]\label{thm 1.1}
$\mathcal{M}^{\mathrm{irr}}_{\mathrm{BaHit}}$ is an empty set or a smooth hyperK\"ahler manifold.
\end{theorem}
The construction of the hyperK\"ahler structure in our setting differs from the classical case. 
In the Riemann surface case, the complex structures underlying the hyperK\"ahler metric are naturally induced by the ordinary Hodge star operator acting on one-forms. 
Since we work on a three-dimensional Sasakian manifold, the ordinary Hodge star does not produce the required complex structures. 
Instead, we use the basic Hodge star operator defined by the transverse K\"ahler metric of the Sasakian structure, which replaces the role of the ordinary Hodge star in the classical case.
See Section \ref{b-pre} for the definition of the basic Hodge star.
\par

As in the Riemann surface case, the basic Hitchin equation is related to flat bundles and Higgs bundles. 
Since Higgs bundles are holomorphic objects, we need basic Higgs bundles.  
We recall this relation in Section \ref{Hi-BH} and \ref{Ha-BH}. 
Hence, we can regard $\mathcal{M}^{\mathrm{irr}}_{\mathrm{BaHit}}$ as a moduli space of simple flat bundles with a fixed basic structure and stable basic Higgs bundles of degree 0.
\par

We also compute the dimension of $\mathcal{M}^{\mathrm{irr}}_{\mathrm{BaHit}}$ under suitable assumptions.

\begin{theorem}[Theorem \ref{dim}]\label{thm 1.2}
Let $(M,(T^{1,0},S,I),(\eta,\xi))$ be a regular Sasakian threefold (See Section \ref{Sasa} for the definition of regular). Let $E$ be a regular basic bundle (See Section \ref{reg bun}) and $h$ be a basic Hermitian metric. Let $g$ be the genus of $M/S^1$. We assume $g\geq 2.$
Then 
\begin{equation*}
 \mathrm{dim}_{\mathbb{R}}\mathcal{M}^{\mathrm{irr}}_{\mathrm{BaHit}}=4(\mathrm{rk}E)^2(g-1)+4.
 \end{equation*}
\end{theorem}
 In the case of a compact Riemann surface of genus $g \ge 2$, the dimension of the Hitchin moduli space was computed in \cite{H,N}.
 Therefore, Theorem \ref{thm 1.2} can be regarded as the Sasakian counterpart of the classical dimension formula for the Hitchin moduli space under suitable assumptions.
 \subsubsection*{Relation to other works}
When $M$ is quasi-regular (See Section \ref{Sasa}), then $M$ is a total space of an $S^1$-bundle over a cyclic orbifold  \cite[Chapter 7]{BG}. 
This is a special case of the Seifelt bundle.
In \cite{BSc}, they study the character variety of the fundamental group of the Seifert bundle. 
From the non-abelian Hodge theory \cite{BH}, the character variety in \cite{BSc} and $\mathcal{M}^{\mathrm{irr}}_{\mathrm{BaHit}}$ should be homeomorphic. The author hopes our space is useful for their work. \par
For the higher-dimensional case, there is a work by Kasuya \cite{Ka}. 
He studied the moduli of flat bundles over general Sasakian manifolds and showed that the moduli have stratification by the basic structure.
He also showed that the moduli space of the flat bundle over a Sasakian threefold is smooth.
While the structure group of his work is $SL(r,\mathbb C)$, we consider $GL(r,\mathbb C)$ in our setting.
However, this difference does not affect the main analytic arguments, and our results can be adapted to the $SL(r,\mathbb C)$ case with only minor modifications. 
In particular, via the non-Abelian Hodge correspondence on a Sasakian manifold, one of the complex structures arising from our hyperK\"ahler structure can be naturally identified with the complex structure considered in his work.
\par
We would also like to mention the connection between the anti-self-dual (ASD) contact instanton and the basic Hitchin equation.
An ASD contact instanton is a solution of a PDE defined on a contact manifold of dimension at least five.
The moduli space in dimension five was studied in \cite{BHa0}, while the seven-dimensional case was investigated in \cite{PE}.
In these higher-dimensional cases, the smoothness problem and the dimension counting become more delicate because obstructions may appear.\par
In the Appendix (Section \ref{sec 5.4}), we show that the basic Hitchin equation arises as a dimensional reduction of the ASD contact instanton equation in five dimensions, which is analogous to the fact that the Hitchin equation arises as a dimensional reduction of the four-dimensional ASD Yang--Mills equation, as pointed out by Hitchin \cite{H}.

\subsubsection*{Acknowledgement} 
The author thanks his supervisor, Hisashi Kasuya, for his enormous support and helpful advice.  
The author thanks M. Benyoussef for explaining their work. The author thanks Akase Kohei for answering his countless elementary questions about analysis. 
The author would like to thank the anonymous referees for their helpful comments that improved the paper.
This work was supported by JSPS KAKENHI Grant Number JP24KJ1611.

\section{Sasakian Manifolds}
\subsection{Sasakian Manifolds}\label{Sasa}
Let $M$ be a (2$n$+1)-dimensional real smooth manifold. Let $TM\otimes \mathbb{C}$ be the complexified tangent bundle of $TM$.  A \textit{CR-structure} on $M$ is a rank $n$ complex sub-bundle $T^{1,0}$ of $TM\otimes \mathbb{C}$ such that $T^{1,0}\cap\overline{T^{1,0}}=0$ and $T^{1,0}$ is integrable. We denote $\overline{T^{1,0}}$ as $T^{0,1}$. For a CR-structure   $T^{1,0}$ on $M$, there is an unique sub-bundle of rank $2n$ of real tangent bundle $TM$ with a vector bundle homomorphism $I:S\to S$ such that the following properties holds:
\begin{itemize}
\item $I^2$ = $-\textrm{Id}_S$,
\item $T^{1,0}$ is the $\sqrt{-1}$-eigen bundle of $I$.
\end{itemize}
A (2$n$+1)-dimensional manifold $M$ is equipped with a triple $(T^{1,0}, S, I)$ is called a \textit{CR-manifold}. A \textit{contact 1-form} $\eta$ of $M$ is a non-degenerate 1-form of $M$ (i.e. $\eta\wedge(d\eta)^{n}$ is everywhere non-zero). By the non-degeneracy of $\eta$, there exists a vector field $\xi$ called \textit{Reeb vector field}  such that it satisfies
\begin{equation*}
\eta(\xi)=1,\xi\lrcorner(d\eta)^{n}=0.
\end{equation*}
  A \textit{contact CR manifold} is a CR-manifold $M$ with a contact 1-form $\eta$ such that \textrm{Ker}$(\eta)=S$. For a contact CR-manifold, the above $I:S\to S$ extends to the entire $TM$ by setting $I(\xi)=0$. Here $\xi$ is the Reeb vector field of $\eta$.
  
\begin{definition}
A contact CR-manifold $(M,(T^{1,0},S,I),(\eta,\xi))$ is a strongly pseudo-convex CR-manifold if the Hermitian form $L_\eta$ on $S_x$ defined by $L_\eta(X,Y)=d\eta(X,IY), X, Y\in S_x$ is positive definite for every point $x\in M$.
\end{definition} 
For a strongly pseudo-convex CR-manifold $(M,(T^{1,0},S,I),(\eta,\xi))$, we have a canonical Riemann metric $g_\eta$ on $M$ which is defined by 
\begin{equation*}
g_\eta(X,Y):=L_\eta(X,Y)+\eta(X)\eta(Y), X,Y\in T_xM.
\end{equation*}
\begin{definition}
A Sasakian manifold is a strongly pseudo-convex CR-manifold 
\begin{equation*}
(M,(T^{1,0},S,I),(\eta,\xi))
\end{equation*}
such that for any section $\zeta$ of $T^{1,0}$, $[\xi,\zeta]$ is also a section of $T^{1,0}$. For a Sasakian manifold, we call $g_\eta$ as Sasaki metric.
\end{definition}
For a Sasakian manifold $(M,(T^{1,0},S,I),(\eta,\xi))$, the metric cone of $(M,g_\eta)$ is a K\"ahler manifold. We can also define a Sasakian manifold as a contact metric manifold whose metric cone is K\"ahler.\par
Let $M$ be a Sasakian manifold. If the orbits of the Reeb vector field $\xi$ are all closed, and hence it is a circle, then $\xi$ induces a $S^1$-action on $M$. Since $\xi$ is nowhere zero, then the action is locally free. We say that $M$ is \textit{regular} if the $S^1$-action is free and \textit{quasi-regular} if it is locally free. When the orbit of $\xi$ is not all closed, then we say $M$ is \textit{irregular}.
\subsection{Basic Differential Forms}\label{b-pre}
Thoughrout this section, let $(M,(T^{1,0},S,I),(\eta,\xi))$ be a  $2n+1$-dimensional compact Sasakian manifold.\par
 The Reeb vector field $\xi$ defines a 1-dimensional foliation $\mathcal{F}_\xi$ on $M$. It is known the map $I:TM\to TM$ associated with the CR-structure $T^{1,0}$ defines a transversely complex structure on the foliated manifold $(M,\mathcal{F}_\xi)$. Furthermore, the closed  2-form $d\eta$ is a transversely K\"ahler structure with respect to this transversely complex structure.\par
 A differential form $\omega$ of $M$ is called a \textit{basic diffrential form} if 
 \begin{equation*}
 i_\xi\omega=0,\mathcal{L}_\xi\omega=0.
\end{equation*}
For simplicity, we call a differential form \textit{basic} if it is a basic differential form. We note that $\eta$ is not basic but $d\eta$ is basic. We denote $A^*_B(M)$ as the space of real basic differential forms. We note that $A^*_B(M)$ forms a sub-complex of deRham complex $A^*(M)$. We denote as $H^i_B(M)$ to be the $i$-th cohomology of $(A^*_B(M),d)$.\par
Corresponding to the decomposition $S_{\mathbb{C}}=T^{1,0}\oplus T^{0,1}$, we have the bigrading 
\begin{equation*}
A^r_B(M)_{\mathbb{C}}=\bigoplus_{p+q=r}A_B^{p,q}(M)
\end{equation*}
as well as the decomposition of the exterior differential
\begin{equation*}
d|_{A^r_B(M)_{\mathbb{C}}}=\partial_\xi+\overline\partial_\xi
\end{equation*}
on $A^r_B(M)_{\mathbb{C}}$, so that 
\begin{align*}
\partial_\xi:A^{p,q}_B(M)&\to A^{p+1,q}_B(M), \\
\overline\partial_\xi:A^{p,q}_B(M)&\to A^{p,q+1}_B(M).
\end{align*}
We also have the transverse Hodge theory (\cite{EKA, KT}). Let 
\begin{equation*}
\ast:A^{r}(M)\to A^{2n+1-r}(M)
\end{equation*}
be the usual Hodge star operator associated with the Sasaki metric $g_\eta$ and let 
\begin{equation*}
\delta:=-\ast d\ast:A^{r}(M)\to A^{r-1}(M)
\end{equation*} 
be the formal adjoint of the exterior derivative with respect to the $L^2$-norm.\par
We define the linear operator 
\begin{equation*}
\star_\xi:A^{r}_B(M)\to A^{2n-r}_B(M)
\end{equation*}
such that $\star_\xi$ acts on $\omega\in A^{r}_B(M)$ as 
\begin{equation*}
\star_\xi\omega=\ast(\eta\wedge\omega).
\end{equation*}
We also define a few more operators:
\begin{align*}
\delta_\xi:=-\star_\xi d\star_\xi: A^r_B(M)&\to A^{r-1}_B(M),\\
\partial_\xi^\ast:=-\star_\xi \overline\partial_\xi\star_\xi: A^{p,q}_B(M)&\to A^{p-1,q}_B(M),\\
\overline\partial_\xi^\ast:=-\star_\xi \partial_\xi\star_\xi: A^{p,q}_B(M)&\to A^{p,q-1}_B(M),\\
\Lambda:=-\star_\xi &\partial_\xi\star_\xi.
\end{align*}
They are the formal adjoints of $d, \partial_\xi,\overline\partial_\xi$ and $d\eta\wedge$ with respect to the pairing
\begin{equation}
A^r_B(M)\times A^r_B(M):(\alpha,\beta)_B\to\int_M\eta\wedge\alpha\wedge\star_\xi\beta. 
\end{equation}
 The following Proposition might be well-known for specialists, however, we give its detailed proof since it is crucial to define the hyperK\"ahler metric for the moduli spaces.
\begin{proposition}\label{cpx str}
Assume $\mathrm{dim}M=3$. Then 
\begin{equation*}
\star_\xi\circ\star_\xi|_{A^1_B(M)}=-\mathrm{Id}_{A^1_B(M)}.
\end{equation*}
\end{proposition}
\begin{proof}
To show the equation holds, it is enough to show it holds pointwise. Let $p\in M$ and $(U,x,y,z)$ be a local coordinate around $p$.
We assume 
\begin{equation*}
S_p=\mathbb{R}\bigg(\frac{\partial}{\partial x}\bigg)_p\oplus \mathbb{R}\bigg(\frac{\partial}{\partial y}\bigg)_p
\end{equation*}
and
\begin{equation*}
\bigg(\frac{\partial}{\partial x}\bigg)_p\bot_{g_\eta}\bigg(\frac{\partial}{\partial y}\bigg)_p\bot_{g_\eta}\xi_p.
\end{equation*}
Under the assumption we have 
\begin{align*}
A^1(M)_p&=\mathbb{R}(dx)_p\oplus \mathbb{R}(dy)_p\oplus\mathbb{R}\eta_p,\\
A^1_B(M)_p&=\mathbb{R}(dx)_p\oplus \mathbb{R}(dy)_p,\\
vol_p&=\eta_p\wedge(dx)_p\wedge(dy)_p.
\end{align*}
Hence we have 
\begin{align*}
\star_\xi(dx)_p&=\ast(\eta_p\wedge(dx)_p)=(dy)_p,\\
\star_\xi(dy)_p&=\ast(\eta_p\wedge(dy)_p)=-(dx)_p.
\end{align*}
Hence the claim is proved.
\end{proof}
\section{Basic Bundles}
 \subsection{Basic Vector Bundles}\label{basic}
 Throughout this section, let $(M,(T^{1,0},S,I),(\eta,\xi))$ be a compact Sasakian manifold.\par
 Let $E$ be a rank $r$ complex vector bundle over $M$. 
 We say that $E$ admits a \textit{basic structure} if there exists a local trivialization $\{U_\alpha\}_{\alpha\in A}$ such that the associated transition functions
$g_{\alpha\beta}: U_\alpha \cap U_\beta \to GL_r(\mathbb{C})$ are basic.
When we say that $E$ is a basic bundle, we mean that $E$ is equipped with a specified basic structure.\par
 Let $E$ be a basic bundle. 
 A $E$-valued differential form $\omega$ is called basic if for every $\alpha\in A$, $\omega|_{U_\alpha}\in A^{p}_B(U_\alpha)\otimes E$. This is well-defined since $E$ is basic. We denote the space of basic $E$-valued $p$-form as $A^{p}_B(E)$. Let $D$ be a connection of $E$. We call $D$ basic if for all $\alpha\in A$, $D|_{U_\alpha}=d+A_\alpha, A_\alpha\in A^1_B(\mathrm{End}E)$.
If $D$ is basic, we have a homomorphism $D: A^\ast_B(E)\to A_B^{\ast+1}(E)$. 
If $D$ is a flat connection, flat frames give a local trivialization of $E$ (See \cite{Ko}).
Since the transition functions of flat frames are locally constant, they define a basic structure on $E$.
 With respect to this basic structure, $D$ is a basic connection.\par

 Let $h$ be a Hermitian metric of $E$. Note that $h\in A(E^\vee\otimes\overline E^\vee)$. Here $E^\vee$ is the dual of $E$. We say the $h$ is basic if $h\in A_B(E^\vee\otimes\overline E^\vee)$. 
 Although Hermitian metrics always exist, basic Hermitian metrics might not exist. 
 However, it exists when $D$ is flat and semi-simple (See Section \ref{sec 5.3}).\par
 We now fix a basic bundle $E$, a basic connection $D$, and assume we have a basic Hermitian metric $h$. 
 As it is well-known $D$ has a decomposition
 \begin{equation}\label{decomp}
 D=\nabla_h+\sqrt{-1}\Phi
 \end{equation}
 such that $\nabla_h$ is a $h$-unitary connection and $\Phi$ is an $\mathrm{End}E$-valued skew-symmetric 1-form w.r.t. $h$. 
 Since $D$ and $h$ are basic, $\nabla_h$ and $\Phi$ are also. 
 We say the $(E,D)$ is \textit{irreducible} if there does not exist a basic sub-bundle $F$ of $E$ with $D(F)\subset A^1_B(F)$. 
 We say $(E,D)$ is \textit{reductive} if $(E,D)$ is a direct sum of irreducible ones. \par
 We define some notations. Let 
 \begin{align*}
 A(\mathfrak{u}(E)):&=\{f\in A(\mathrm{End}E):h(fu,v)+h(u,fv)=0\},\\
 A_r(\mathfrak{u}(E)):&=\{f\in  A(\mathfrak{u}(E)):\int_M\mathrm{tr}(f)=0\},\\
  A^i(\mathfrak{u}(E)):&=A^i\otimes A(\mathfrak{u}(E)),\\
 A^i_{r}(\mathfrak{u}(E)):&=A^i\otimes A_r(\mathfrak{u}(E)),\\
  A^i_B(\mathfrak{u}(E)):&=A^i_B\otimes A(\mathfrak{u}(E)),\\
 A^i_{B,r}(\mathfrak{u}(E)):&=A^i_B\otimes A_r(\mathfrak{u}(E)).
   \end{align*}
   We say $A_r(\mathfrak{u}(E))$ (resp. $A_{B,r}(\mathfrak{u}(E))$) as (basic) reduced section. We note that we have the following $L^2$-decomposition.
   \begin{align*}
    A(\mathfrak{u}(E))&= A_r(\mathfrak{u}(E))\oplus \sqrt{-1}\mathbb{R}\mathrm{Id}_E,\\
    A_B(\mathfrak{u}(E))&= A_{B,r}(\mathfrak{u}(E))\oplus \sqrt{-1}\mathbb{R}\mathrm{Id}_E.   
     \end{align*}
 Let $D$ be a basic connection. 
 It is well-known that $D$ admits a decomposition
 \begin{equation}\label{decomp}
 D=\nabla_h+\sqrt{-1}\Phi
 \end{equation}
 such that $\nabla_h$ is an $h$-unitary connection and $\Phi$ is an $\mathrm{End}E$-valued skew-symmetric 1-form w.r.t. $h$. 
 Since $D$ and $h$ are basic, $\nabla_h$ and $\Phi$ are also basic. 
We say that the pair $(\nabla_h,\Phi)$ is \textit{simple} if there does not exist a proper basic subbundle $F$ with $\nabla_h(F)\subset A^1_B(F),\Phi(F)\subset A^1_B(F)$.
We use the term "simple" to distinguish it from the notion of irreducibility for the connection $D=\nabla_h+\sqrt{-1}\Phi$ (see Section \ref{BaHit}).
\par
  The following result is used for the calculation of the dimension of the moduli space. 
 Although it follows from standard arguments in gauge theory, we give the proof for completeness.
  \begin{proposition}\label{irr} 
  The following are equivalent.
  \begin{itemize}
  \item $(\nabla_h,\Phi)$ is simple.
   \item We define a differential operator $D_1:A_B(\mathfrak{u}(E))\to A^1_B(\mathfrak{u}(E))\oplus A^1_B(\mathfrak{u}(E))$ as follows:
   \begin{equation*}
   D_1(f):=(\nabla_hf,[\Phi,f]).
   \end{equation*}  
   Then $\mathrm{Ker}(D_1)=\sqrt{-1}\mathbb{R}\mathrm{Id}_E$.
   \end{itemize}
   \end{proposition}
   \begin{proof}
  Assume $(\nabla_h,\Phi)$ is simple. 
Suppose we have a $f\in A_B(\mathfrak{u}(E))\backslash \sqrt{-1}\mathbb{R}\mathrm{Id}_E$ such that $D_1f=0$. By the definition of $D_1$, we have $\nabla_hf=0$. From \cite[p.25, Proposition 1.1.17]{LT}, we have the eigendecomposition of $E$ with respect to $f$:
   \begin{equation*}
   E=\bigoplus_\lambda E_\lambda.
   \end{equation*}
 Since $f$ is basic, each $E_\lambda$ is basic.
  The decomposition is $h$-othogonal and $\nabla_h(E_\lambda)\subset A^1_B(E_\lambda)$.
   Since each  $E_\lambda$ is eigen bundle of $f$ and $\Phi f-f\Phi=[\Phi,f]=0$, we have $\Phi(E_\lambda)\subset A^1_B(E_\lambda)$.  
   This contradicts the simplicity of $(\nabla_h,\Phi)$.\par
 Assume  $\mathrm{Ker}(D_1)=\sqrt{-1}\mathbb{R}\mathrm{Id}_E$. 
 Suppose $(\nabla_h,\Phi)$ is not simple. 
 Then there exists a sub-bundle $E_1\subset E$ such that $\nabla_h(E_1)\subset A^1_B(E_1), \Phi(E_1)\subset A^1_B(E_1)$.
 Let $E_2:=(E_1)^\perp$ be the orthogonal bundle of $E_1$.
 We have a $h$-orthogonal decomposition
 \[
 E=E_1\oplus E_2.
 \]
 Since $\nabla_h$ is a $h$-unitary connection, $\Phi$ is an $\mathrm{End}E$-valued skew-symmetric 1-form w.r.t. $h$, and $E_2$ is the othogonal bundle of $E_1$, $\nabla_h$ and $\Phi$ preserves $E_2$.\par
 Let $pr_1$ and $pr_2$ be the orthogonal projections to $E_1$ and $E_2$. 
 By definition, $\sqrt{-1}pr_1,\sqrt{-1}pr_2\in A_B(\mathfrak{u}(E))$. 
It is straightforward to check $\sqrt{-1}pr_1-\sqrt{-1}pr_2\in A_B(\mathfrak{u}(E))\backslash \sqrt{-1}\mathbb{R}\mathrm{Id}_E$ and
\begin{align*}
\nabla_h(\sqrt{-1}pr_1-\sqrt{-1}pr_2)=0, [\Phi,\sqrt{-1}pr_1-\sqrt{-1}pr_2]=0.
\end{align*} 
This contradicts  $\mathrm{Ker}(D_1)=\sqrt{-1}\mathbb{R}\mathrm{Id}_E$.
     \end{proof}
     
    \begin{remark}
    In \cite{BHa}, the authors defined that a basic metric connection $\nabla_h$ is irreducible if $\mathrm{Ker}(\nabla_h)|_{A_B(\mathfrak{u}(E))}=\sqrt{-1}\mathbb{R}\mathrm{Id}_E$.
     Proposition \ref{irr} tells us that the definition of our simplicity and their irreducibility coincide when $\Phi=0$.
    \end{remark}
    
\begin{remark}
Proposition \ref{irr} recovers \cite[Proposition 1.1.17 ]{LT} in the case $\Phi=0$,
and may be regarded as its extension to the setting with a Higgs field.
\end{remark}
    
    Let $A_B(GL(E))$ be the automorphism group of the basic bundle $E$. We define the gauge group
    \begin{equation*}
    \mathcal{G}_B:=\{ f\in A_B(GL(E)): h(fu,fv)=h(u,v)\}.
    \end{equation*}
    We moreover define the reduced gauge group as
    \begin{equation*}
    \mathcal{G}_{B,r}:=\mathcal{G}_B\big/S^1\mathrm{Id}_E.
        \end{equation*}
   The Lie algebra of $\mathcal{G}_B$ is $A(\mathfrak{u}(E))$ and $\mathcal{G}_{B,r}$ is $A_r(\mathfrak{u}(E))$.\par
   Let $\mathcal{A}_{h,B}$ be the space of basic $h$-unitary connections.
    This is an affine space that is modeled on $A^1_B(\mathfrak{u}(E))$.
    We define 
     \begin{equation*}
     \mathcal{A}_{B}:= \mathcal{A}_{h,B}\times A^1_B(\mathfrak{u}(E)).
               \end{equation*}
    Since any basic connection $D$ has the decomposition (\ref{decomp}), we regard $\mathcal{A}_{B}$ as the space of connections. $\mathcal{G}_B(E)$  acts on $\mathcal{A}^k_{B}$ as  
    \begin{equation}
    \begin{split}
    \mathcal{G}_B\times \mathcal{A}_{B}&\longrightarrow  \mathcal{A}_{B}\\
    (g,\nabla_h,\Phi)&\longmapsto (g^{-1}\nabla_h g,g^{-1}\Phi g).
   \end{split}
         \end{equation}

    \subsubsection{$L^2$-metric, Adjoints, and Brackets}
    In this section, we review some operations around $A^i_B(\mathfrak{u}(E))$. 
    The results in this section are nothing new. 
    However, we write this section to make the paper clear.\par
    Let $(E,h)$ be a basic vector bundle with a basic Hermitian metric on a Sasakian manifold $M$. Let $A,B\in A^i(\mathrm{End}E)$. Recall that the $L^2$-inner product $(A,B)_{L^2}$ is defined as 
    \begin{equation*}
    (A,B)_{L^2}=\int_M\mathrm{Tr}(A\wedge\ast B^\dagger_h).
    \end{equation*}
    Here recall that $B^\dagger_h$ is the formal adjoint of $B$ w.r.t. $h$ and $\ast$ is the ordinary Hodge star. Hence if we assume $B\in A^i(\mathfrak{u}(E))$, we have 
    \begin{equation*}
     (A,B)_{L^2}=\int_M\mathrm{Tr}(A\wedge\ast B^\dagger_h)=- \int_M\mathrm{Tr}(A\wedge\ast B).   
    \end{equation*}
    We study the $L^2$-metric restricted to $A^i_B(\mathfrak{u}(E))$. Let $\alpha\in A^i(M)$. The usual Hodge star $\ast$ and the basic Hodge star $\star_\xi$ have the following relation (\cite{KT}):
    \begin{equation*}
    \ast\alpha=\star_\xi\alpha\wedge\eta.
    \end{equation*} 
  Hence if $A,B\in A^i_B(\mathfrak{u}(E))$, we have 
  \begin{equation*}
  (A,B)_{L^2}=- \int_M\mathrm{Tr}(A\wedge\ast B)=- \int_M\mathrm{Tr}(A\wedge\star_\xi B)\wedge\eta.
      \end{equation*}
 Let $\nabla_h\in \mathcal{A}_{h,B}$ and $\Phi\in A^1_{B}(\mathfrak{u}(E))$. Let $\nabla^\ast_h$ and $\Phi^\ast$ be the formal adjoints of $\nabla_h$ and $\Phi$ w.r.t. the $L^2$-inner product i.e. for $A\in A^{i}(\mathrm{End}(E))$ and $B\in A^{i+1}(\mathrm{End}(E))$, the following holds
 \begin{align*}
 (\nabla_h A,B)_{L^2}&=(A,\nabla^\ast_h B)_{L^2},\\
 ([\Phi,A],B)_{L^2}&=(A,[\Phi^\ast,B])_{L^2}.
 \end{align*}
 We give the explicit formula of $\nabla^\ast_h$ and $\Phi^\ast$ when we restrict the $L^2$-inner product to $A^\ast_B(\mathfrak{u}(E))$.
Since the Sasakian manifold has no basic $2n+1$-form, for  $A\in A^{i}_B(\mathfrak{u}(E))$ and $B\in A^{i+1}_B(\mathfrak{u}(E))$, we have 
\begin{align*}
 (\nabla_h A,B)_{L^2}&=(A,\nabla^\ast_h B)_{L^2}=-(A,\star_\xi\nabla_h\star_\xi B)_{L^2},\\
 ([\Phi,A],B)_{L^2}&=(A,[\Phi^\ast,B])_{L^2}=(A,\star_\xi[\Phi^\dagger_h,\star_\xi B])_{L^2}=-(A,\star_\xi[\Phi,\star_\xi B])_{L^2}.
 \end{align*}
Hence we have 
\begin{lemma}\label{f-ad}
When we restrict the action of $\nabla_h$ and $\Phi$ to $A^\ast_B(\mathfrak{u}(E))$, those formal adjoints $\nabla^\ast_h,\Phi^\ast$ w.r.t. the $L^2$-inner product has the form
\begin{align*}
 \nabla^\ast_h&=-\star_\xi\nabla_h\star_\xi,\\
\Phi^\ast&=-\star_\xi\Phi\star_\xi.
 \end{align*}
 \end{lemma}
We state one more result, which we use later. From now on, we assume $\mathrm{dim}M=3$.
\begin{lemma}\label{bra}
Let $A,B\in A^1_B(\mathrm{End}E)$. Then
\begin{equation*}
[\star_\xi A,B]=-[A,\star_\xi B]
\end{equation*}
holds.
\end{lemma}
\begin{proof}
We only have to prove it pointwise. Let $p\in M$. We use the coordinate which we used in Proposition \ref{cpx str}. 
\begin{align*}
[\star_\xi A,B]_p&=[\star_\xi(A_x(dx)_p+A_y(dy)_p),B_x(dx)_p+B_y(dy)_p]\\
&=[A_x(dy)_p-A_y(dx)_p,B_x(dx)_p+B_y(dy)_p]\\
&=-\bigg([A_x,B_x]+[A_y,B_y]\bigg)(dx)_p\wedge(dy)_p.
\end{align*}
\begin{align*}
[A,\star_\xi B]_p&=[A_x(dx)_p+A_y(dy)_p,\star_\xi(B_x(dx)_p+B_y(dy)_p)]\\
&=[A_x(dx)_p+A_y(dy)_p,B_x(dy)_p-B_y(dx)_p]\\
&=\bigg([A_x,B_x]+[A_y,B_y]\bigg)(dx)_p\wedge(dy)_p.
\end{align*}
Hence, the Lemma is proved.
\end{proof}

\subsubsection{Regular Bundles}\label{reg bun}
In this section, we recall the notion of regular bundles on Sasakian manifolds based on \cite{BH2}.
We now assume the Sasakian manifold $(M,(T^{1,0},S,I),(\eta,\xi))$ is regular.
 Let
 \begin{equation*}
 \varphi:\mathbb{R}\times M \to M, (t,x)\mapsto \varphi_t(x)
 \end{equation*}
 be the flow generated by the Reeb vector field. \par
 Let $E$ be a basic vector bundle of rank $r$. Then by \cite{BH2}, we can define a natural action
 \begin{equation*}
 \Phi:\mathbb{R}\times E \to E, (t,e)\mapsto \Phi_t(e)
 \end{equation*}
 such that they are compatible with the natural projection $p_E:E\to M$ (i.e. $p_E\circ\Phi_t=\varphi_t$). 
 Since $M$ is regular, the flow $\varphi:\mathbb{R}\times M \to M$ induces a free smooth action $\psi:S^1\times M\to M$.
  This is equivalent to the existence of a positive number $r\in\mathbb{R}$ such that $\varphi_r(x)=1$ for all $x\in M$. 
  The minimum of all such $r_{Min}$ is called the period of $M$. 
  We assume  $r_{Min}=1$ for simplicity. 
  \begin{definition}
 We say that $E$ is \textit{quasiregular} if $\Phi:\mathbb{R}\times E\to E$ induces a $S^1$-action
 $\Psi:S^1\times E\to E$. 
 This is equivalent to the existence of a positive integer $m$ such that $\Phi_m=\mathrm{Id}_E$. 
 We say that $E$ is \textit{regular} if it is quasi-regular and $m=1$.
  \end{definition}

  \section{The Moduli space of Basic Hitchin Equations}
 Throughout this section, we assume $(M,(T^{1,0},S,I),(\eta,\xi))$ to be a compact Sasakian manifold of dimension three. 
 We also fix a basic bundle $E$ and a basic metric $h$.

 \subsection{Basic Hitchin Equation}\label{BaHit}
Recall that we defined  $\mathcal{A}_{h,B}$ to be the set of basic $h$-unitary connections and $A^1_B(\mathfrak{u}(E))$ be $\mathrm{End}E$-valued skew-symmetric 1-forms (See section \ref{basic}) with respect to $h$.
Note that $ \mathcal{A}_{h,B}$ is an affine space modeled on $A^1_B(\mathfrak{u}(E))$.
 \par
 Let $(\nabla_h,\Phi)\in  \mathcal{A}_{B}=\mathcal{A}_{h,B}\times A^1_B(\mathfrak{u}(E))$. We say that $(\nabla_h,\Phi)$ satisfies the \textit{basic Hitchin equation} if 
 \begin{equation}\label{BHeq}
 \left\{
 \begin{split}
 F_{\nabla_h}-\Phi\wedge\Phi&=0,\\
\nabla_h\Phi&=0,\\
\nabla_h\star_\xi\Phi&=0.
 \end{split}
 \right. 
  \end{equation}
  Here $F_{\nabla_h}$ is the curvature of $\nabla_h$.
  If $(\nabla_h,\Phi)$ satisfies the Hitchin equation we call $(\nabla_h,\Phi)$ a \textit{basic Hitchin pair}. 
  We set as
  \begin{equation*}
  \mathcal{A}_{\mathrm{BaHit}}:=\{ (\nabla_h,\Phi)\in\mathcal{A}_{h,B}\times A^1_B(\mathfrak{u}(E)): (\nabla_h,\Phi) \textrm{ is a basic Hitchin pair}\}.
    \end{equation*}
    Let $D$ be a basic connection.
  We say that $(E, D)$ is \textit{irreducible} if there does not exist a proper basic subbundle $E_1\subset E$ such that $D(E_1)\subset A^1_B(E_1)$.
  When the underlying bundle $E$ is fixed, we simply say that $D$ is irreducible.\par
   We set
    \begin{equation*}
 \mathcal{A}^{\mathrm{irr}}_{\mathrm{BaHit}}:=\{ (\nabla_h,\Phi)\in  \mathcal{A}_{\mathrm{BaHit}}: D=\nabla_h+\sqrt{-1}\Phi \textrm{ is irreducible}\}.
    \end{equation*}  
 Recall that in Section \ref{basic}, we defined a notion \textit{simple} for a pair $(\nabla_h,\Phi)\in  \mathcal{A}_{B}=\mathcal{A}_{h,B}\times A^1_B(\mathfrak{u}(E))$.
 We show that simplicity of the pair  $(\nabla_h,\Phi)\in  \mathcal{A}_{B}=\mathcal{A}_{h,B}\times A^1_B(\mathfrak{u}(E))$
 is equivalent to the irreducibility of $D$ if $(\nabla_h,\Phi)\in \mathcal{A}_{\mathrm{BaHit}}$ by the non-Abelian Hodge correspondence of Sasakian manifolds (See Section \ref{sec 5}).
 We also note that if $(\nabla_h,\Phi)\in \mathcal{A}_{\mathrm{BaHit}}$, then $(\nabla^{0,1}_h,\sqrt{-1}\Phi^{1,0})$ is a polystable basic Higgs bundle (See Section \ref{sec 5}).
 
  \begin{lemma}\label{lem 4.1}
 Let $(\nabla_h,\Phi)\in \mathcal{A}_{\mathrm{BaHit}}$.
 Then the following are equivalent.
 \begin{itemize}
 \item[1.] $(\nabla_h,\Phi)$ is simple.
 \item[2.] $(E, \nabla^{0,1}_h,\sqrt{-1}\Phi)$ is stable.
 \item[3.] $(E,D=\nabla_h+\sqrt{-1}\Phi)$ is irreducible.
   \end{itemize}
  \end{lemma}
\begin{proof}
$3.\to 1.$ Suppose $(\nabla_h,\Phi)$ is not simple. 
Then there exists a basic subbundle $E_1\subset E$ such that $\nabla_h(E_1),\Phi(E_1)\subset A^1_B(E_1)$.
This implies $D(E_1)\subset A^1_B(E_1)$ which contradicts the irreducibility of $D$.\par
$1.\to 2.$ By the result of \cite{BH, BH2,WZ}, $(\nabla^{0,1}_h,\sqrt{-1}\Phi)$ is polystable.
Assume that it is not stable.
Let $V\subset E$ be a basic sub-Higgs bundle such that it is stable and has the same slope as $E$ (See Section \ref{sec 5}).
Then by \cite[Proposition 3.3]{S1}, $E=V\oplus V^\perp$ is an orthogonal direct sum of basic sub-Higgs bundles.
Since the decomposition is orthogonal, $\nabla_h(V),\Phi(V)\subset A^1_B(V)$.
This contradicts the simplicity of $(\nabla_h,\Phi)$.\par
$2.\to 3.$ Suppose $(E,D=\nabla+\sqrt{-1}\Phi)$ is not irreducible.
Then we have a basic sub-bundle $D(V)\subset A^1_B(V)$.
Since $D$ is flat and a connection of $V$, the degree of $V$ is 0 and this contradicts the stability of $(E,\nabla^{0,1}_h,\sqrt{-1}\Phi)$.
\end{proof}
   
     Note that the action of the gauge groups $\mathcal{G}_B$ and $\mathcal{G}_{B,r}$ preserves $\mathcal{A}_{\mathrm{BaHit}}$ and $\mathcal{A}^{\mathrm{irr}}_{\mathrm{BaHit}}$. Moreover, $\mathcal{G}_{B,r}$ acts freely on $\mathcal{A}^{\mathrm{irr}}_{\mathrm{BaHit}}$.\par
 Let $(\nabla_h,\Phi)\in  \mathcal{A}_{\mathrm{BaHit}}$.
  Considering the linearization of the action of the gauge group $\mathcal{G}_B$ and the linearization of the basic Hitchin equation (\ref{BHeq}), we obtain a complex
 \begin{equation}\label{cpx}
 \begin{split}
 0\longrightarrow A_B(\mathfrak{u}(E))\overset{D_1}{\longrightarrow} A^1_B(\mathfrak{u}(E))^{\oplus 2}\overset{D_2}{\longrightarrow} A^2_B(\mathfrak{u}(E))^{\oplus 3}\longrightarrow 0
    \end{split}
 \end{equation}
 where
 \begin{equation}\label{cpx 1}
 \begin{split}
D_1A&:=(\nabla_hA,[\Phi,A]),\\
D_2(A,B)&:=(\nabla_hA-[\Phi,B],\nabla_hB+[A,\Phi],\nabla_h\star_\xi B+[A,\star_\xi\Phi]).
    \end{split}
 \end{equation}
 \par
 Note that $D_1$ is exactly the same operator we introduced in Proposition \ref{irr}. Considering the highest-order part of the differential operators $D_1$ and $D_2$, we see that the complex (\ref{cpx}) is \textit{transverse elliptic complex} (See \cite{Wa}). We denote the $i$-th cohomology of the complex (\ref{cpx}) as $\mathbb{H}^i$. These cohomology are finite dimensions since they are the kernel of transverse elliptic operators \cite{EKA}. The dimension of $\mathbb{H}^1$ is expected to be the dimension of the moduli space. \par
 We now consider the case $(\nabla_h,\Phi)\in \mathcal{A}^{\mathrm{irr}}_{\mathrm{BaHit}}$. In this case, $\mathrm{Ker}D_1=\sqrt{-1}\mathbb{R}\mathrm{Id}_E$ (See Proposition \ref{irr} and Lemma \ref{lem 4.1}) and hence $\mathrm{dim}_{\mathbb{R}}\mathbb{H}^0=1$. We later use the following result to show the moduli space is smooth and to calculate the dimension of the moduli space.
\begin{proposition}\label{co 2nd} 
Assume $(\nabla_h,\Phi)\in \mathcal{A}^{\mathrm{irr}}_{\mathrm{BaHit}}$. Then
$\mathrm{dim}_{\mathbb{R}}\mathbb{H}^2=3.$ In particular each row of $\mathbb{H}^2$ is spanned by the multiplication of $\sqrt{-1}d\eta$ and $\mathrm{Id}_E$ i.e. 
\begin{equation*}
\mathbb{H}^2=[\langle \sqrt{-1}d\eta\mathrm{Id}_E\rangle^{\oplus 3}_{\mathbb{R}}].
\end{equation*}
Here 
\begin{align*}
&\langle \sqrt{-1}d\eta\mathrm{Id}_E\rangle^{\oplus 3}_{\mathbb{R}}:=
\mathbb{R}
\begin{pmatrix}
\sqrt{-1}d\eta\mathrm{Id}_E\\
0\\
0
\end{pmatrix}
+
\mathbb{R}
\begin{pmatrix}
0\\
\sqrt{-1}d\eta\mathrm{Id}_E\\
0
\end{pmatrix}
+
\mathbb{R}
\begin{pmatrix}
0\\
0\\
\sqrt{-1}d\eta\mathrm{Id}_E
\end{pmatrix},
\end{align*}
and $[\langle \sqrt{-1}d\eta\mathrm{Id}_E\rangle^{\oplus 3}_{\mathbb{R}}]$ 
is the $\mathbb{R}$-vector space which is spanned by the cohomology class of the basis of $\langle \sqrt{-1}d\eta\mathrm{Id}_E\rangle^{\oplus 3}_{\mathbb{R}}$.
\end{proposition}
\begin{proof} 
It is enough to show  
\begin{equation*}
\mathrm{Ker}D^\ast_2=
\langle\sqrt{-1}d\eta\mathrm{Id}_E\rangle^{\oplus 3}_{\mathbb{R}}
\end{equation*}
Let $(A,B,C)\in A^2_B(\mathfrak{u}(E))^{\oplus 3}$. By direct calculation, we have 
\begin{equation*}
D^\ast_2(A,B,C)=(\nabla^\ast_hA+[(\star_\xi\Phi)^\ast,B]+[\Phi^\ast, C],-[\Phi^\ast,A]-\star_\xi\nabla_h^\ast B+\nabla_h^\ast C).
\end{equation*}
Here $\nabla^\ast_h$ is the formal adjoint of $\nabla_h$ w.r.t. $L^2$-inner product. $\Phi^\ast$, $(\star_\xi\Phi)^\ast$ are also. \par

Hence $D^\ast_2(A,B,C)=0$ is equivalent to 
\begin{equation}\label{ad0}
\left\{
\begin{split}
\nabla^\ast_hA+[(\star_\xi\Phi)^\ast,B]+[\Phi^\ast, C]&=0,\\
-[\Phi^\ast,A]-\star_\xi\nabla_h^\ast B+\nabla_h^\ast C&=0.
\end{split}
\right.
\end{equation}
Recall that from Lemma \ref{f-ad}, we have the explicit formula of $\nabla^\ast_h$, $\Phi^\ast$, and $(\Phi^{1,0})^\ast$:\begin{align*}
\nabla^\ast_h&=-\star_\xi\nabla_h\star_\xi,\\
(\Phi)^\ast&= \star_\xi(\Phi)_h^\dagger\star_\xi=-\star_\xi\Phi\star_\xi,\\
(\star_\xi\Phi)^\ast&=\star_\xi(\star_\xi\Phi)_h^\dagger\star_\xi=-\star_\xi(\star_\xi\Phi)\star_\xi.
\end{align*}
\par
The operator $\star_\xi$ induces an isomorphism
\begin{equation*}
\star_\xi:A^2_B(\mathfrak{u}(E))\to A_B(\mathfrak{u}(E)).
\end{equation*} 
Hence to consider  the pair  $(A,B,C)\in A^2_B(\mathfrak{u}(E))^{\oplus 3}$ which satisfies the equation (\ref{ad0}) is equivalent to consider the pair
 $(\alpha,\beta,\gamma)\in A_B(\mathfrak{u}(E))^{\oplus 3}$ which satisfies the following equations
\begin{equation}\label{ad1}
\left\{
\begin{split}
\nabla_h\alpha+[\star_\xi\Phi,\beta]+[\Phi,\gamma]&=0,\\
[\Phi,\alpha]+\star_\xi\nabla_h \beta-\nabla_h \gamma&=0.
\end{split}
\right.
\end{equation}
Let $(,)_{L^2}$ be the $L^2$-inner product. Assume $(\alpha,\beta,\gamma)\in A_B(\mathfrak{u}(E))^{\oplus 3}$ satisfies the equation (\ref{ad1}). Then we have 
\begin{align*}
\|\nabla_h\alpha\|^2_{L^2}&=(\nabla_h\alpha,\nabla_h\alpha)_{L^2}\\
&=(-\star_\xi\nabla_h\star_\xi\nabla_h\alpha,\alpha)_{L^2}\\
&=(\star_\xi\nabla_h\star_\xi[\star_\xi\Phi,\beta]+\star_\xi\nabla_h\star_\xi[\Phi,\gamma],\alpha)\\
&=(-\star_\xi\nabla_h[\Phi,\beta]+\star_\xi\nabla_h[\star_\xi\Phi,\gamma],\alpha)_{L^2} \qquad (\because \:\mathrm{Lemma}\: \ref{bra}.)\\
&=(\star_\xi[\Phi,\nabla_h\beta]-\star_\xi[\star_\xi\Phi,\nabla_h\gamma],\alpha)_{L^2}\\
&=(\star_\xi[\Phi,\nabla_h\beta]+\star_\xi[\Phi,\star_\xi\nabla_h\gamma],\alpha)_{L^2}\\
&=(\star_\xi[\Phi,\star_\xi(-\star_\xi\nabla_h\beta+\nabla_h\gamma)],\alpha)_{L^2}\\
&=(\star_\xi[\Phi,\star_\xi[\Phi,\alpha]],\alpha)_{L^2}\\
&=-((\Phi)^\ast[\Phi,\alpha],\alpha)_{L^2}\\
&=-([\Phi,\alpha],[\Phi,\alpha])_{L^2}\\
&=-\|[\Phi,\alpha]\|^2_{L^2}.
\end{align*}
Hence we obtain $\nabla_h\alpha=[\Phi,\alpha]=0$. This is equivalent to $\alpha\in \mathrm{Ker}D_1$. Since $(\nabla_h,\Phi)\in \mathcal{A}^{\mathrm{irr}}_{\mathrm{BaHit}}$, $\alpha=\sqrt{-1}a\mathrm{Id}_E$ for some $a\in\mathbb{R}$. Then $\beta$ and $\gamma$ satisfies 
\begin{equation}\label{ad2}
\left\{
\begin{split}
[\star_\xi\Phi,\beta]+[\Phi,\gamma]&=0,\\
\star_\xi\nabla_h \beta-\nabla_h \gamma&=0.
\end{split}
\right.
\end{equation}
We first calculate $\|\nabla_h\gamma\|^2_{L^2}$.
\begin{equation*}
\begin{split}
\|\nabla_h\gamma\|^2_{L^2}&=(\nabla_h\gamma,\nabla_h\gamma)_{L^2}\\
&=-(\star_\xi\nabla_h\star_\xi\nabla_h\gamma,\gamma)_{L^2}\\
&=-(\star_\xi\nabla_h\star_\xi\star_\xi\nabla_h \beta,\gamma)_{L^2}\\
&=(\star_\xi\nabla_h\nabla_h \beta,\gamma)_{L^2}\\
&=(\star_\xi F_{\nabla_h}\beta,\gamma)_{L^2}\\
&=(\star_\xi[\Phi,[\Phi,\beta]],\gamma)_{L^2}\\
&=-(\star_\xi[\Phi,\star_\xi\star_\xi[\Phi,\beta]],\gamma)_{L^2}\\
&=((\Phi)^\ast\star_\xi[\Phi,\beta],\gamma)_{L^2}\\
&=([\star_\xi\Phi,\beta],[\Phi,\gamma])_{L^2}\\
&=-([\star_\xi\Phi,\beta],[\star_\xi\Phi,\beta])_{L^2}\\
&=-\|[\star_\xi\Phi,\beta]\|^2_{L^2}.
\end{split}
\end{equation*}
Hence we obtain $\nabla_h\gamma=[\star_\xi\Phi,\beta]=0$. Since $\beta$ and $\gamma$ satisfies the equation (\ref{ad2}), we also obtain $\star_\xi\nabla_h\beta=[\Phi,\gamma]=0$. Since $\star_\xi$ is an isomorphism, $\nabla_h\Phi=[\Phi,\beta]=0$. Hence $\beta,\gamma\in\mathrm{Ker}D_1$, and therefore $\beta=\sqrt{-1}b\mathrm{Id}_E$ and $\gamma=\sqrt{-1}c\mathrm{Id}_E$ for some $b,c\in\mathbb{R}$.
\par
Let $(A,B,C)\in\mathrm{Ker}D^\ast_2$. Then $(\alpha,\beta,\gamma):=(\star_\xi A,\star_\xi B,\star_\xi C)$ satisfies the equation (\ref{ad1}). By the discussion above, $(\alpha,\beta,\gamma)=(\sqrt{-1}a\mathrm{Id}_E,\sqrt{-1}b\mathrm{Id}_E,\sqrt{-1}c\mathrm{Id}_E)$ for some $a,b,c\in \mathbb{R}$. Since we have $\star_\xi 1=d\eta$, $A,B,C\in\langle\sqrt{-1}d\eta\mathrm{Id}_E\rangle_{\mathbb{R}}.$ Hence $\mathrm{ker}D^\ast_2\subset\langle \sqrt{-1}d\eta\mathrm{Id}_E\rangle^{\oplus 3}_{\mathbb{R}}$.
\par
Since $\star_\xi d\eta=1$, $\langle \sqrt{-1}d\eta\mathrm{Id}_E\rangle^{\oplus 3}_{\mathbb{R}}\subset \mathrm{ker}D^\ast_2$. Hence we have
\begin{equation*}
\mathrm{ker}D^\ast_2=\langle \sqrt{-1}d\eta\mathrm{Id}_E\rangle^{\oplus 3}_{\mathbb{R}}.
\end{equation*}
\end{proof}  \par
We now construct the moduli space of the irreducible basic Hitchin pair. To construct the moduli space, we introduce $\|\cdot\|_{k,2}$ the $L^2_k$-Sobolev norm. Let $L^2_k(A^1_B(\mathfrak{u}(E)))$ to be the completion of $A^1_B(\mathfrak{u}(E))$ with respect to the $L^2_k$-norm. 
We denote as $\mathcal{A}^{k}_{h,B}$ to be the space of basic metric $L^2_k$- connection.
We set 
\begin{equation*}
\mathcal{A}^k_{B}:=\mathcal{A}^{k}_{h,B}\times L^2_k(A^1_B(\mathfrak{u}(E))).
\end{equation*}  
 We may regard $\mathcal{A}^k_{B}$ as the space of basic $L^2_k$-connection.
Let   $\mathcal{G}^{k}_B$ to be the $L^2_k$-basic gauge group and $\mathcal{G}^{k}_{r,B}:=\mathcal{G}^{k}_B/S^1\mathrm{Id_E}$ to be the reduced $L^2_k$-basic gauge group. We take $k$ large enough so that the basic Sobolev embedding holds \cite{BHa, KLW}. Then one can show as in \cite{DK}, that $\mathcal{G}^{k}_B$ and $\mathcal{G}^{k}_{r,B}$ are Hilbert Lie groups. By basic Sobolev multiplication \cite{BHa, KLW}, $\mathcal{G}^{k+1}_B$ and $\mathcal{G}^{k+1}_{r,B}$ acts smoothly on $\mathcal{A}^k_{B}$ and we can show that $\mathcal{B}^k:=\mathcal{A}^k_{B}/\mathcal{G}^{k+1}_B$ and $\mathcal{B}^k_r:=\mathcal{A}^k_{B}/\mathcal{G}^{k+1}_{r,B}$ are  Hausdorff spaces in the quotient topology. Let $\mathcal{A}^k_{\mathrm{BaHit}}\subset \mathcal{A}^k_{B}$ be the space of $L^2_k$-basic Hitchin pair. We define the moduli space of $L^2_k$-basic Hitchin equation $\mathcal{M}^k_{\mathrm{BaHit}}$ as 
\begin{equation*}
\mathcal{M}^k_{\mathrm{BaHit}}:=\mathcal{A}^k_{\mathrm{BaHit}}/\mathcal{G}^{k+1}_{r,B}.
\end{equation*}
Since  $\mathcal{M}^k_{\mathrm{BaHit}}\subset \mathcal{B}^k_r$, $\mathcal{M}^k_{\mathrm{BaHit}}$ is a Hausdorff space. We define $\mathcal{A}^{k,\mathrm
{irr}}_{B}\subset \mathcal{A}^k_{B}$ to be the space irreducible basic $L^2_k$-connection and $\mathcal{A}^{k,\mathrm{irr}}_{\mathrm{BaHit}}:=\mathcal{A}^k_{\mathrm{BaHit}}\cap\mathcal{A}^{k,\mathrm{irr}}_{B}$ to be the space of irreducible basic $L^2_k$-Hitchin pairs. Note that $\mathcal{G}^{k+1}_{r,B}$ acts freely on $\mathcal{A}^{k,\mathrm
{irr}}_{B}$ and $\mathcal{A}^{k,\mathrm{irr}}_{\mathrm{BaHit}}$. We define $\mathcal{B}^{k,\mathrm{irr}}_r:=\mathcal{A}^{k,\mathrm{irr}}_B/ \mathcal{G}^{k+1}_{r,B}$. We finally define the moduli of irreducible $L^2_k$-basic Hitchin pairs as 
   \begin{equation*}
    \mathcal{M}^{k,\mathrm{irr}}_{\mathrm{BaHit}}:=\mathcal{A}^{k,\mathrm{irr}}_{\mathrm{BaHit}}/\mathcal{G}^{k+1}_{r,B}.
        \end{equation*}
  Since $\mathcal{B}^{k,\mathrm{irr}}_r\subset\mathcal{B}^{k}_r$ and $\mathcal{M}^{k,\mathrm{irr}}_{\mathrm{BaHit}}\subset \mathcal{M}^{k}_{\mathrm{BaHit}}$, they are Hausdorff spaces. The topology of $\mathcal{M}^{k,\mathrm{irr}}_{\mathrm{BaHit}}$ do depend on $k$. However, we can apply the argument in \cite{DK, LT} and show the following.
  \begin{proposition} Assume that $k$ is large enough. Then the natural map $\mathcal{M}^{k+1,\mathrm{irr}}_{\mathrm{BaHit}}\to\mathcal{M}^{k,\mathrm{irr}}_{\mathrm{BaHit}}$ is a homeomorphism.
 \end{proposition}
Since we have this Proposition, we omit the subscription $k$ from now.\par
  We now turn our interest to the local structure of the moduli space. Let $[(\nabla_h,\Phi)]\in \mathcal{B}^{\mathrm{irr}}_r$. We define a slice 
  \begin{equation}
  S_{(\nabla_h,\Phi),\epsilon}:=\{\alpha\in A^1_B(\mathfrak{u}(E))^{\oplus 2}:\|\alpha\|_{L^2_k}<\epsilon, D^\ast_1\alpha=0\}.
  \end{equation}
  We can apply the argument of \cite{DK,LT,Pa} and  show that $S_{(\nabla_h,\Phi),\epsilon}$ gives a coordinate patch for $\mathcal{B}^{irr}_r$. \par
  From now on, we assume $[(\nabla_h,\Phi)]\in  \mathcal{M}^{\mathrm{irr}}_{\mathrm{BaHit}}$. 
  We show that $ \mathcal{M}^{\mathrm{irr}}_{\mathrm{BaHit}}\cap S_{(\nabla_h,\Phi),\epsilon}$ is diffeomorphic to the neighborhood of $\mathbb{H}^1$. 
  Before we proceed, we prepare some notations.
   We set $\Delta_{i,(\nabla_h,\Phi)}:=D_{i}D^\ast_{i}+D^\ast_{i+1}D_{i+1} (i=0,1,2)$ to be the Laplacians. 
  We set as $D_{-1}=D_3=0$.
   Let $G_{(\nabla_h,\Phi)}$ be the Green operators and $H_{(\nabla_h,\Phi)}$ be the Harmonic projections. 
   We denote as $\Delta_i,G,H$ if there is no confusion.  \par
  Let $\alpha=(A,B)\in S_{(\nabla_h,\Phi),\epsilon}$. Then $\alpha\in  \mathcal{M}^{\mathrm{irr}}_{\mathrm{BaHit}}$ if and only if 
  \begin{equation}\label{kur 0}
  \begin{split}
  D_2\alpha+
   \begin{pmatrix}
  A\wedge A-B\wedge B\\
  [A,B]\\
  [A,\star_\xi B]
  \end{pmatrix}=
    D_2(A,B)+
  \begin{pmatrix}
  A\wedge A-B\wedge B\\
  [A,B]\\
  [A,\star_\xi B]
  \end{pmatrix}=0.
  \end{split}
  \end{equation}
  This can be checked by direct computation. 
  To simplify the notation, we set
  \begin{equation*}
\widetilde{\alpha\wedge\alpha}:=
   \begin{pmatrix}
  A\wedge A-B\wedge B\\
  [A,B]\\
  [A,\star_\xi B]
  \end{pmatrix}.  
  \end{equation*}
 Note that $\widetilde{\alpha\wedge\alpha}$ is not an ordinary wedge product.\par
  Hence we have 
  \begin{equation*}
 \mathcal{M}^{\mathrm{irr}}_{\mathrm{BaHit}}\cap S_{(\nabla_h,\Phi),\epsilon}=\{\alpha\in S_{(\nabla_h,\Phi),\epsilon}:  D_2\alpha+\alpha\wedge\alpha=0\}.
    \end{equation*}
  By the Hodge decomposition, the equation (\ref{kur 0}) is equivalent to  
  \begin{equation}\label{kur 1}
  \left\{
  \begin{split}
 & D_2\alpha+D_2D^\ast_2G(\widetilde{\alpha\wedge\alpha})=0,\\
  & H(\widetilde{\alpha\wedge\alpha})=0.
  \end{split}
  \right.
    \end{equation}
    We define the \textit{Kuranishi map} $k_{(\nabla_h,\Phi)}:A^1_B(\mathfrak{u}(E))^{\oplus2}\to A^1_B(\mathfrak{u}(E))^{\oplus2}$ as 
    \begin{equation}\label{kur-map}
    k_{(\nabla_h,\Phi)}(\alpha)=\alpha+D^\ast_2G(\alpha\wedge\alpha).
  \end{equation}
  Let $\alpha\in  \mathcal{M}^{\mathrm{irr}}_{\mathrm{BaHit}}\cap S_{(\nabla_h,\Phi),\epsilon}$. Then by (\ref{kur 1}),
  \begin{align*}
  D^\ast_1(k_{(\nabla_h,\Phi)}(\alpha))&=D^\ast_1\alpha +D^\ast_1D^\ast_2G(\alpha\wedge\alpha)=0,\\
  D_2(k_{(\nabla_h,\Phi)}(\alpha))&=D_2\alpha+D_2D^\ast_2G(\alpha\wedge\alpha)=0.
    \end{align*}
    Hence 
 \begin{equation*}
 k_{(\nabla_h,\Phi)}(\mathcal{M}^{\mathrm{irr}}_{\mathrm{BaHit}}\cap S_{(\nabla_h,\Phi),\epsilon})\subset \mathbb{H}^1.
  \end{equation*} 
  The next proposition shows that $\mathcal{M}^{\mathrm{irr}}_{\mathrm{BaHit}}$ is smooth.
 \begin{proposition}
 Let $U$ be a neighborhood of the origin of  $\mathbb{H}^1$. If we take a $U$ small enough, then there exists a $\epsilon$ such that $k_{(\nabla_h,\Phi)}$ induces a homeomorphism 
 \begin{equation*}
 k_{(\nabla_h,\Phi)}:\mathcal{M}^{\mathrm{irr}}_{\mathrm{BaHit}}\cap S_{(\nabla_h,\Phi),\epsilon}\to U.
   \end{equation*}
 \end{proposition} 
 \begin{proof}
The proof is quite standard (See \cite{Ko}). 
The point of this proposition is that we do not need any assumption to show $\mathcal{M}^{\mathrm{irr}}_{\mathrm{BaHit}}$ is smooth. \par
Let $L^2_k(A^1_B(\mathfrak{u}(E)))$ be the completion of $A^1_B(\mathfrak{u}(E))$ with respect to the $L^2_k$-norm. 
We extend the Kuranishi map to
\begin{equation*}
k_{(\nabla_h,\Phi)}:L^2_k(A^1_B(\mathfrak{u}(E)))^{\oplus2}\to L^2_k(A^1_B(\mathfrak{u}(E)))^{\oplus2}.
\end{equation*}
Since the derivative of the Kuranishi map at the origin is the identity, we can apply the inverse function theorem of Banach spaces and show that there exist neighborhoods of the origin  $V_1$ and $V_2$ such that $k_{(\nabla_h,\Phi)}$ induces a homeomorphism
\begin{equation*}
k_{(\nabla_h,\Phi)}:V_1\to V_2.
\end{equation*}
Let $\beta\in V_2\cap\mathbb{H}^1$. Let $\alpha:=k^{-1}(\beta)$. We show that $\alpha\in V_1\cap\mathrm{Ker}D^\ast_1\cap\mathcal{M}^{k,\mathrm{irr}}_{\mathrm{BaHit}}.$ Once this is shown, shrink $V_1$ and we prove the proposition.\par
 First, from the definition of $\alpha$, we have 
\begin{equation*}
\beta=\alpha+D^\ast_2G(\widetilde{\alpha\wedge\alpha}).
\end{equation*}
Act the Laplacian $\Delta_1$ and we have 
\begin{align*}
0=\Delta_1\beta&=\Delta_1\alpha+D^\ast_2\Delta_2G(\widetilde{\alpha\wedge\alpha})\\
&=\Delta_1\alpha+D^\ast_2\Delta_2G(\widetilde{\alpha\wedge\alpha})\\
&=\Delta_1\alpha+D^\ast_2(\widetilde{\alpha\wedge\alpha})-D^\ast_2H(\widetilde{\alpha\wedge\alpha})\\
&=\Delta_1\alpha+D^\ast_2(\widetilde{\alpha\wedge\alpha}).
\end{align*}
Hence, by the transverse elliptic regularity, $\alpha$ is smooth. We also have
\begin{align*}
0&=D_2\beta=D_2\alpha+D_2D^\ast_2G(\widetilde{\alpha\wedge\alpha}),\\
0&=D^\ast_1\beta=D^\ast_1\alpha.
\end{align*}
We now showed that $\alpha\in V_1\cap\mathrm{Ker}D^\ast_1$. To show $\alpha\in\mathcal{M}^{\mathrm{irr}}_{\mathrm{BaHit}}$, we need to show $H(\alpha\wedge\alpha)=0$ (See (\ref{kur 1})). To show this, we use Proposition  \ref{co 2nd}. Recall that 
 \begin{equation*}
\widetilde{\alpha\wedge\alpha}=
   \begin{pmatrix}
  A\wedge A-B\wedge B\\
  [A,B]\\
  [A,\star_\xi B]
  \end{pmatrix}.  
  \end{equation*}
 From Proposition \ref{co 2nd}, there exists $a,b,c\in\mathbb{R}$ such that 
  \begin{align*}
  H
  \begin{pmatrix}
   A\wedge A-B\wedge B\\
  [A,B]\\
  [A,\star_\xi B]
    \end{pmatrix}
    &=
    \sqrt{-1}
 \begin{pmatrix}
a \\
b\\
c\\
 \end{pmatrix}d\eta\mathrm{Id}_E.\\
  \end{align*}  
 We would like to show $a=b=c=0$. First, let 
 \begin{equation*}
A^i_B(\mathfrak{su}(E)):=\{f\in A^i_B(\mathfrak{u}(E)): \mathrm{Tr}(f)=0\}.
\end{equation*}
Then the complex 
 \begin{equation*}
  \begin{split}
 0\longrightarrow A_B(\mathfrak{su}(E))\overset{D_1}{\longrightarrow} A^1_B(\mathfrak{su}(E))^{\oplus 2}\overset{D_2}{\longrightarrow} A^2_B(\mathfrak{su}(E))^{\oplus 3}\longrightarrow 0
    \end{split}
 \end{equation*} 
 forms a sub complex of (\ref{cpx}). Since 
 \begin{equation*}
\begin{pmatrix}
   A\wedge A-B\wedge B\\
  [A,B]\\
  [A,\star_\xi B]
    \end{pmatrix}
    \in A^2_B(\mathfrak{su}(E))^{\oplus 3},
    \end{equation*} 
  we have 
\begin{equation*}
  H
\begin{pmatrix}
   A\wedge A-B\wedge B\\
  [A,B]\\
  [A,\star_\xi B]
    \end{pmatrix}
    \in \mathbb{H}^2\cap A^2_B(\mathfrak{su}(E))^{\oplus 3}.
    \end{equation*}
Hence $\mathrm{Tr}(a\cdot d\eta\mathrm{Id}_E)=\mathrm{Tr}(b\cdot d\eta\mathrm{Id}_E)=\mathrm{Tr}(c\cdot d\eta\mathrm{Id}_E)$=0. We obtain $a=b=c=0$.
 \end{proof}
 In particular, we have the following
 \begin{corollary}\label{m-sm}
$\mathcal{M}^{\mathrm{irr}}_{\mathrm{BaHit}}$ is an empty set or a smooth manifold. 
If not empty,  the dimension of $\mathcal{M}^{\mathrm{irr}}_{\mathrm{BaHit}}$ around $[(\nabla_h,\Phi)]\in \mathcal{M}^{\mathrm{irr}}_{\mathrm{BaHit}}$ is $\mathbb{H}^1$.
 \end{corollary}
We give a sufficient condition for  $\mathcal{M}^{\mathrm{irr}}_{\mathrm{BaHit}}$ not to be empty. 
Recall that $T^{1,0}$ is the CR structure on $M$. 
If $c_{1,B}(T^{1,0})=-C[d\eta]$ for some positive constant $C$, then there exists a basic stable Higgs bundle due to \cite[Example 3.6]{BH2}. 
Hence if $c_{1,B}(T^{1,0})=-C[d\eta], C>0$, then $\mathcal{M}^{\mathrm{irr}}_{\mathrm{BaHit}}$ is not empty (See Section \ref{Hi-BH}).
 \subsection{Riemannian Structure on $\mathcal{M}^{\mathrm{irr}}_{\mathrm{BaHit}}$}
We use the same notation of the previous section.
 We assume that $\mathcal{M}^{\mathrm{irr}}_{\mathrm{BaHit}}$ is not an empty set. \par
We show that the moduli space $\mathcal{M}^{\mathrm{irr}}_{\mathrm{BaHit}}$ of irreducible basic Hitchin pair on a compact Sasakian threefold $M$ is a hyperK\"ahler manifold. 
We first define a Riemannian metric $g$ on $\mathcal{M}^{
\mathrm{irr}}_{\mathrm{BaHit}}$. Let $[(\nabla_h,\Phi)]\in \mathcal{M}^{\mathrm{irr}}_{\mathrm{BaHit}}$ and $\alpha=(\alpha_1,\alpha_2), \beta=(\beta_1,\beta_2)\in \mathbb{H}^1\simeq T_{[(\nabla_h,\Phi)]}\mathcal{M}^{\mathrm{irr}}_{\mathrm{BaHit}}$. We define $g$ as 
\begin{equation}\label{metric}
g_{[(\nabla_h,\Phi)]}(\alpha,\beta):=-\int_M\mathrm{Tr}(\alpha_1\wedge\star_\xi \beta_1+\alpha_2\wedge\star_\xi \beta_2)\wedge\eta.
\end{equation} 
To show $g$ is well-defined, we need to check that $g$ does not depend on the gauge-equivalence class of $[(\nabla_h,\Phi)]\in\mathcal{M}^{irr}_{\mathrm{BaHit}}$. Under a gauge transformation $(\nabla_h,\Phi)\to h^{-1}(\nabla_h,\Phi)h$, the infinitesimal deformations $\alpha,\beta$ maps to $h^{-1}\alpha h, h^{-1}\beta h$ which are the corresponding harmonic repsentative (See \cite{It} for details.). Since (\ref{metric}), the metric $g$ is equivalent to the gauge transformation. Hence $g$ is well-defined.\par
We now prove the distinguished coordinate of the moduli $\mathcal{M}^{\mathrm{irr}}_{\mathrm{BaHit}}$ induced by the Kuranishi map and the slice is a normal coordinate with respect to $(\mathcal{M}^{\mathrm{irr}}_{\mathrm{BaHit}},g)$. 
This result will be used later to show that $\mathcal{M}^{\mathrm{irr}}_{\mathrm{BaHit}}$ is hyperK\"ahler.\par
Let $[(\nabla_h,\Phi)]\in \mathcal{M}^{\mathrm{irr}}_{\mathrm{BaHit}}$. Then from the previous section we have the Kuranishi map $k_{(\nabla_h,\Phi)}$, Slice $S_{(\nabla_h,\Phi),\epsilon}$, and a open subset  $0\in U\subset \mathbb{H}^1$ such that 
\begin{equation*}
k_{(\nabla_h,\Phi)}:\mathcal{M}^{\mathrm{irr}}_{\mathrm{BaHit}}\cap S_{(\nabla_h,\Phi),\epsilon}\to U
\end{equation*}
is a homeomorphism. The derivative of the Kuranishi map at $\alpha\in A^1_{B}(\mathfrak{u}(E))^{\oplus 2}$ as follows
\begin{equation}\label{der-kur}
\begin{split}
&d(k_{(\nabla_h,\Phi)})_\alpha:T_\alpha A^1_{B}(\mathfrak{u}(E))^{\oplus 2}\to T_{k_{(\nabla_h,\Phi)}(\alpha)} A^1_{B}(\mathfrak{u}(E))^{\oplus 2},\\
&d(k_{(\nabla_h,\Phi)})_\alpha(\beta)=\beta+D^\ast_2G(\widetilde{[\alpha,\beta]}).
\end{split}
\end{equation}
Here for $\alpha=(\alpha_1,\alpha_2),\beta=(\beta_1,\beta_2)\in A^1_{B}(\mathfrak{u}(E))^{\oplus 2}$ we defined $\widetilde{[\alpha,\beta]}$ as 
\begin{equation}
\widetilde{[\alpha,\beta]}:=
\begin{pmatrix}
[\alpha_1,\beta_1]-[\alpha_2,\beta_2]\\
[\alpha_1,\beta_2]+[\beta_1,\alpha_2]\\
[\alpha_1,\star_\xi \beta_2]+[\beta_1,\star_\xi\alpha_2]
\end{pmatrix}.
\end{equation}
Note that $\widetilde{[\alpha,\beta]}$ is not the ordinary bracket. We call this bracket as the \textit{modified bracket}.\par
Using the modified bracket, we can characterize the tangent space of $\alpha\in \mathcal{M}^{\mathrm{irr}}_{\mathrm{BaHit}}\cap S_{(\nabla_h,\Phi),\epsilon}$ as follows
\begin{equation}\label{tan-mod}
T_\alpha (\mathcal{M}^{\mathrm{irr}}_{\mathrm{BaHit}}\cap S_{(\nabla_h,\Phi),\epsilon})=\{\beta\in A^1_{B}(\mathfrak{u}(E))^{\oplus 2}: D^\ast_1\beta=0, D_2\beta+\widetilde{[\alpha,\beta]}=D_{2,\alpha}\beta=0\}.
\end{equation}
Here $D_{2,\alpha}$ is the operator of (\ref{cpx 1}) defined for $(\nabla_h,\Phi)+\alpha=(\nabla_h+\alpha_1,\Phi+\alpha_2)\in \mathcal{A}^{\mathrm{irr}}_{\mathrm{BaHit}}$. From (\ref{der-kur}) and (\ref{tan-mod}), the restriction of $dk_{(\nabla_h,\Phi)}$ to $T_\alpha(\mathcal{M}^{\mathrm{irr}}_{\mathrm{BaHit}}\cap S_{(\nabla_h,\Phi),\epsilon})$ has the following form.
\begin{proposition}\label{der-kur2}
The differential of the Kuranishi map 
\begin{equation*}
d(k_{(\nabla_h,\Phi)})_\alpha:T_\alpha (\mathcal{M}^{\mathrm{irr}}_{\mathrm{BaHit}}\cap S_{(\nabla_h,\Phi),\epsilon})\to T_{k_{(\nabla_h,\Phi)}(\alpha)}U=\mathbb{H}^1
\end{equation*}
has the form
\begin{equation*}
d(k_{(\nabla_h,\Phi)})_\alpha(\beta)=H_{(\nabla_h,\Phi)}\beta.
\end{equation*}
Here $H_{(\nabla_h,\Phi)}:A^1_B(\mathfrak{u}(E))^{\oplus 2}\to\mathbb{H}^1$ is the harmonic projection.
\end{proposition}
\begin{proof}
Since $D^\ast_2$ commutes with the Green operator, and we have  (\ref{der-kur}) and (\ref{tan-mod}), we have
\begin{align*}
d(k_{(\nabla_h,\Phi)})_\alpha(\beta)&=\beta+D^\ast_2G(\widetilde{[\alpha,\beta]})\\
&=\beta-D^\ast_2GD_2\beta\\
&=\beta-D^\ast_2D_2G\beta\\
&=H_{(\nabla_h,\Phi)}\beta.
\end{align*}
\end{proof}
In the previous section, we denoted $H_{(\nabla_h,\Phi)}$ just as $H$. We denoted as $H_{(\nabla_h,\Phi)}$ because later, we use the harmonic projection induced by different basic Hitchin pairs.\par
We now solve conversely an equation $d(k_{(\nabla_h,\Phi)})_\alpha(\beta)=\gamma$ for a given $\gamma\in T_{k_{(\nabla_h,\Phi)}(\alpha)}U=\mathbb{H}^1$ and $\alpha\in\mathcal{M}^{\mathrm{irr}}_{\mathrm{BaHit}}\cap S_{(\nabla_h,\Phi),\epsilon}$ with respect to $\beta\in T_\alpha (\mathcal{M}^{\mathrm{irr}}_{\mathrm{BaHit}}\cap S_{(\nabla_h,\Phi),\epsilon})$. We decompose $\beta$ as 
\begin{equation*}
\beta=D_1\gamma_0+\gamma_1+D^\ast_2\gamma_2,
\end{equation*} 
where $\gamma_0\in A_B(\mathfrak{u}(E))$, $\gamma_1\in \mathbb{H}^1$, and $\gamma_2\in A^2_B(\mathfrak{u}(E))^{\oplus 3}$.
By Proposition \ref{der-kur2}, $\gamma_1=\gamma$.
Moreover, since $D^\ast_1\beta=0,$ we have $D^\ast_1D_1\gamma_0=0$ and hence $D_1\gamma_0=0$.
Hence we obtain
\begin{equation*}
\beta=\gamma+D^\ast_2\gamma_2.
\end{equation*} 
From (\ref{tan-mod}), $\gamma_2$ satisfies the equation
\begin{equation*}
D_2D^\ast_2\gamma_2+\widetilde{[\alpha,\gamma+D^\ast_2\gamma_2]}=0.
\end{equation*} 
By the definition of the modified bracket, it is a bilinear map.
Hence 
\begin{equation}\label{kur-inv}
D_2D^\ast_2\gamma_2+\widetilde{[\alpha,D^\ast_2\gamma_2]}=-\widetilde{[\alpha,\gamma]}.
\end{equation} 
As a consequence we have
\begin{proposition}\label{kur-inv 1}
For a given $\gamma\in\mathbb{H}^1$, the inverse image $\beta= (d(k_{(\nabla_h,\Phi)})_\alpha)^{-1}(\gamma)\in T_\alpha (\mathcal{M}^{\mathrm{irr}}_{\mathrm{BaHit}}\cap S_{(\nabla_h,\Phi),\epsilon})$ is represented by
\begin{equation*}
\beta=\gamma+D^\ast_2\gamma_2
\end{equation*} 
where $\gamma_2\in A^2_B(\mathfrak{u}(E))^{\oplus 3}$ is a solution of (\ref{kur-inv}).
\end{proposition}
We note that at the origin, $T_0 (\mathcal{M}^{\mathrm{irr}}_{\mathrm{BaHit}}\cap S_{(\nabla_h,\Phi),\epsilon})=\mathbb{H}^1$ and $d(k_{(\nabla_h,\Phi)})_0=\mathrm{Id}_{\mathbb{H}^1}$ holds.\par
Let $X,Y, Z\in T_0 (\mathcal{M}^{\mathrm{irr}}_{\mathrm{BaHit}}\cap S_{(\nabla_h,\Phi),\epsilon})=\mathbb{H}^1$.
Since $\mathbb{H}^1$ is affine, these vectors also define vector fields on $U$ canonically.
We define a vector field $\overline X$ on  $\mathcal{M}^{\mathrm{irr}}_{\mathrm{BaHit}}\cap S_{(\nabla_h,\Phi),\epsilon}$ as 
\begin{equation*}
\overline X_\alpha:=d((k_{(\nabla_h,\Phi)})^{-1})_{k_{(\nabla_h,\Phi)}(\alpha)}(X), \, \alpha\in\mathcal{M}^{\mathrm{irr}}_{\mathrm{BaHit}}\cap S_{(\nabla_h,\Phi),\epsilon}.
\end{equation*}
We define $\overline Y,\overline Z$ in the same manner.
From Proposition \ref{kur-inv 1}, $\overline X_\alpha$ has the form 
\begin{equation*}
\overline X_\alpha=X+D^\ast_2\gamma(\alpha,X)
\end{equation*}
where $\gamma(\alpha,X)\in  A^2_B(\mathfrak{u}(E))^{\oplus 3}$ and it satisfies the following equation
\begin{equation}\label{gam-eq}
D_2D^\ast_2\gamma(\alpha,X)+\widetilde{[\alpha,D^\ast_2\gamma(\alpha,X)]}=-\widetilde{[\alpha,X]}.
\end{equation}
We note that at $\alpha=0$, $\overline X_0=X$ and $D^\ast_2\gamma(0,X)=0$.\par
Let $c(t)$ be a curve on $\mathcal{M}^{\mathrm{irr}}_{\mathrm{BaHit}}\cap S_{(\nabla_h,\Phi),\epsilon}$ defined by $c(t):=(k_{(\nabla_h,\Phi)})^{-1}(tX)$.
Then we have $c(0)=0$ and $\frac{d}{dt}c(t)|_{t=0}=X\in T_0 (\mathcal{M}^{\mathrm{irr}}_{\mathrm{BaHit}}\cap S_{(\nabla_h,\Phi),\epsilon})=\mathbb{H}^1$.
\begin{proposition}\label{der-m}
The Riemannian metric $g$ on $\mathcal{M}^{\mathrm{irr}}_{\mathrm{BaHit}}$ satisfies at $\alpha=0$ in a slice neighborhood $\mathcal{M}^{\mathrm{irr}}_{\mathrm{BaHit}}\cap S_{(\nabla_h,\Phi),\epsilon}$
\begin{equation*}
Xg_{[(\nabla_h,\Phi)]}(Y,Z)=0
\end{equation*}
for every $X,Y,Z\in T_0 (\mathcal{M}^{\mathrm{irr}}_{\mathrm{BaHit}}\cap S_{(\nabla_h,\Phi),\epsilon})=\mathbb{H}^1.$
\end{proposition}
We remark that this Proposition shows that the coordinate obtained by the Kuranishi map is normal.
\begin{proof}
By the definition of the metric 
\begin{align*}
Xg_{[(\nabla_h,\Phi)]}(Y,Z)&=\frac{d}{dt}g_{[(\nabla_h,\Phi)+c(t)]}(\overline Y_{c(t)}, \overline Z_{c(t)})\bigg|_{t=0}\\
&=\frac{d}{dt}\bigg(H_{(\nabla_h,\Phi)+c(t)}\overline Y_{c(t)},H_{(\nabla_h,\Phi)+c(t)} \overline Z_{c(t)}\bigg)_{L^2}\bigg|_{t=0}\\
&=\bigg(\frac{d}{dt}(H_{(\nabla_h,\Phi)+c(t)}\overline Y_{c(t)})|_{t=0},Z\bigg)_{L^2}+\bigg(Y,\frac{d}{dt}(H_{(\nabla_h,\Phi)+c(t)} \overline Z_{c(t)})\bigg|_{t=0}\bigg)_{L^2}.
\end{align*}
Differentiating $H_{(\nabla_h,\Phi)+c(t)}\overline Y_{c(t)}$ at $t=0$, we get
\begin{align*}
\frac{d}{dt}\bigg(H_{(\nabla_h,\Phi)+c(t)}\overline Y_{c(t)}\bigg)\bigg|_{t=0}=\bigg(\frac{d}{dt}H_{(\nabla_h,\Phi)+c(t)}\bigg|_{t=0}\bigg)Y+H_{(\nabla_h,\Phi)}\bigg(\frac{d}{dt}\overline Y_{c(t)}\bigg|_{t=0}\bigg).
\end{align*}
Before we proceed, we prepare two Lemmas.
\begin{lemma}\label{der-m1}
\begin{equation*}
H_{(\nabla_h,\Phi)}\bigg(\frac{d}{dt}\overline Y_{c(t)}\bigg|_{t=0}\bigg)=0.
\end{equation*}
\end{lemma}
\begin{proof}
From Proposition \ref{kur-inv 1}, we have 
\begin{align*}
\frac{d}{dt}\overline Y_{c(t)}\bigg|_{t=0}&=\frac{d}{dt}(Y+D^\ast_2\gamma(c(t),Y))\bigg|_{t=0}\\
&=D^\ast_2\bigg(\frac{d}{dt}\gamma(c(t),Y)\bigg|_{t=0}\bigg).
\end{align*}
From (\ref{gam-eq}), $\gamma(c(t),Y)$ satisfies the equation
\begin{equation*}
D_2D^\ast_2\gamma(c(t),Y)+\widetilde{[c(t),D^\ast_2\gamma(c(t),Y)]}=-\widetilde{[c(t),Y]}.
\end{equation*}
We differential this equation at $t=0$ and we obtain 
\begin{equation*}
D_2D^\ast_2\bigg(\frac{d}{dt}\gamma(c(t),Y)\bigg|_{t=0}\bigg)=-\widetilde{[X,Y]}.
\end{equation*}
By Proposition \ref{co 2nd} and the Hodge decomposition, we have $a,b,c\in\mathbb{R}$ such that 
\begin{align*}
\frac{d}{dt}\gamma(c(t),Y)|_{t=0}&=
\sqrt{-1}
\begin{pmatrix}
a\\
b\\
c\\
\end{pmatrix}
d\eta+GD_2D^\ast_2\bigg(\frac{d}{dt}\gamma(c(t),Y)|_{t=0}\bigg)\\
&=
\sqrt{-1}
\begin{pmatrix}
a\\
b\\
c\\
\end{pmatrix}
d\eta-\widetilde{[X,Y]}.
\end{align*}
Then we have 
\begin{align*}
\frac{d}{dt}\overline Y_{c(t)}|_{t=0}&=D^\ast_2\bigg(\frac{d}{dt}\gamma(c(t),Y)\bigg|_{t=0}\bigg)\\
&=D^\ast_2\bigg(\sqrt{-1}
\begin{pmatrix}
a\\
b\\
c\\
\end{pmatrix}
d\eta-G\widetilde{[X,Y]}\bigg)\\
&=-D^\ast_2G\widetilde{[X,Y]}.
\end{align*}
Then the Lemma is obtained by the Hodge decomposition.
\end{proof}
\begin{lemma}\label{der-m2}
\begin{align*}
\bigg(\frac{d}{dt}H_{(\nabla_h,\Phi)+c(t)}|_{t=0}\bigg)Y&=-G[X,D^\ast_1Y]^1-D_1G[X,Y]^2-D^\ast_2G\widetilde{[X,Y]}-G[X,D_2Y]^3\\
&=-D_1G[X,Y]^2-D^\ast_2G\widetilde{[X,Y]}.
\end{align*}
Here 
\begin{align*}
[X,D^\ast_1Y]^1:&=
\begin{pmatrix}
[X_1,D^\ast_1 Y]\\
[X_2,D^\ast_1 Y]
\end{pmatrix},\\
[X,Y]^2:&=[X^\ast_1, Y_1]+[X^\ast_2, Y_2],\\
[X,D_2 Y]^3:&=
\begin{pmatrix}
[X_1,\nabla_h Y_1-[\Phi,Y_2]]+[\star_\xi X^\ast_2,\nabla_h Y_2+[Y_1,\Phi]]+[X^\ast_2,\nabla_h\ast Y_2+[Y_1,\ast\Phi]]\\
-[X^\ast_2,\nabla_hY_2+[Y_1,\Phi]]-[\star_\xi X^\ast_2,\nabla_hY_1-[\Phi,Y_2]]+[X^\ast_2,\nabla_h\star_\xi Y_2+[Y_1,\star_\xi\Phi]]
\end{pmatrix}.
\end{align*}
\end{lemma}
\begin{proof}
The second equality follows from the harmonicity of $Y$. 
We prove the first equality.\par
By the Hodge decomposition, we have
\begin{align*}
\bigg(\frac{d}{dt}H_{(\nabla_h,\Phi)+c(t)}\bigg|_{t=0}\bigg)Y&=\frac{d}{dt}(H_{(\nabla_h,\Phi)+c(t)}Y)\bigg|_{t=0}\\
&=-\frac{d}{dt}(G_{c(t)}\Delta_{1,(\nabla_h,\Phi)+c(t)}Y)\bigg|_{t=0}\\
&=-\frac{d}{dt}\bigg(G_{c(t)}\bigg|_{t=0}\bigg)\Delta_{1,(\nabla_h,\Phi)}Y+G\frac{d}{dt}(\Delta_{1,(\nabla_h,\Phi)+c(t)}Y)\bigg|_{t=0}\\
&=-G\frac{d}{dt}(\Delta_{1,(\nabla_h,\Phi)+c(t)}Y)\bigg|_{t=0}.
\end{align*}
We now calculate $\frac{d}{dt}(\Delta_{1,(\nabla_h,\Phi)+c(t)}Y)|_{t=0}$.
\begin{align*}
\frac{d}{dt}(\Delta_{1,(\nabla_h,\Phi)+c(t)}Y)\bigg|_{t=0}&=\frac{d}{dt}(D_{1,(\nabla_h,\Phi)+c(t)}D^\ast_{1,(\nabla_h,\Phi)+c(t)}Y+D^\ast_{2,(\nabla_h,\Phi)+c(t)}D_{2,(\nabla_h,\Phi)+c(t)}Y)\bigg|_{t=0}\\
&=[X,D^\ast_1Y]^1+D_1[X,Y]^2+D^\ast_2\widetilde{[X,Y]}+[X,D_2Y]^3\\
&=D_1[X,Y]^2+D^\ast_2\widetilde{[X,Y]}.
\end{align*}
Hence, the claim is proved.
\end{proof}
We now prove the Proposition.
From the two Lemmas above, we have 
\begin{align*}
Xg_{[(\nabla_h,\Phi)]}(Y,Z)&=\bigg(\frac{d}{dt}(H_{(\nabla_h,\Phi)+c(t)}\overline Y_{c(t)})\bigg|_{t=0},Z\bigg)_{L^2}+\bigg(Y,\frac{d}{dt}(H_{(\nabla_h,\Phi)+c(t)} \overline Z_{c(t)})\bigg|_{t=0}\bigg)_{L^2}\\
&=\bigg(-D_1G[X,Y]_2-D^\ast_2G\widetilde{[X,Y]},Z\bigg)_{L^2}+\bigg(Y,-D_1G[X,Z]_2-D^\ast_2G\widetilde{[X,Z]}\bigg)_{L^2}\\
&=0.
\end{align*}
The last follows from the harmonicity of $Y$ and $Z$.
\end{proof}
\subsection{HyperK\"ahler Structure on $\mathcal{M}^{\mathrm{irr}}_{\mathrm{BaHit}}$}
We use the same notation as the previous section. We assume that $\mathcal{M}^{\mathrm{irr}}_{\mathrm{BaHit}}$ is not an empty set.\par
We define almost complex structures $\mathcal{I},\mathcal{J},\mathcal{K}$ on $\mathcal{M}^{\mathrm{irr}}_{\mathrm{BaHit}}$.
We first fix a $(\nabla_h,\Phi)\in\mathcal{A}_{\mathrm{BaHit}}$.
First, we show that $A^1_B(\mathfrak{u}(E))^{\oplus 2}$ has the structure of a quaternion vector space.
Next, we show that they induce a quaternion structure to $\mathbb{H}^1$.  \par
Let $\alpha=(\alpha_1,\alpha_2)\in A^1_B(\mathfrak{u}(E))^{\oplus 2}$. We define $I,J,K\in \mathrm{End}(A^1_B(\mathfrak{u}(E))^{\oplus 2})$ as follows
\begin{align*}
I
\begin{pmatrix}
\alpha_1\\
\alpha_2
\end{pmatrix}
&:=
\begin{pmatrix}
\star_\xi \alpha_1\\
-\star_\xi \alpha_2
\end{pmatrix},\\
J
\begin{pmatrix}
\alpha_1\\
\alpha_2
\end{pmatrix}
&:=
\begin{pmatrix}
-\alpha_2\\
\alpha_1
\end{pmatrix},\\
K
\begin{pmatrix}
\alpha_1\\
\alpha_2
\end{pmatrix}
&:=
\begin{pmatrix}
-\star_\xi \alpha_2\\
-\star_\xi \alpha_1
\end{pmatrix}.
\end{align*}
By Proposition \ref{cpx str} and definition of $I,J,$ and, $K$ we can check that 
\begin{align*}
I^2=J^2=K^2=-\mathrm{Id}, \, K=IJ
\end{align*}
and hence $I,J,K$ defines a quaternion structure of $A^1_B(\mathfrak{u}(E))^{\oplus 2}$.
To show that $I,J,K$ induces a quaternion structure to $\mathbb{H}^1$, we only need to check that $I,J,K$ preserves $\mathrm{Ker}D^\ast_1\cap\mathrm{Ker}D_2$.
This can be shown by direct computation.
Note that for $\alpha=(\alpha_1,\alpha_2)\in A^1_B(\mathfrak{u}(E))^{\oplus 2}$, we have
\begin{equation}\label{dual D1}
\begin{split}
D^\ast_1\alpha&=\nabla^\ast_h\alpha_1+\Phi^\ast\alpha_2\\
&=-\star_\xi\nabla_h\star_\xi\alpha_1-\star_\xi[\Phi,\star_\xi\alpha_2].
\end{split}
\end{equation}
Hence by (\ref{cpx 1}) and (\ref{dual D1}),  $\alpha\in\mathrm{Ker}D^\ast_1\cap\mathrm{Ker}D_2$ if and only if the following equations hold
\begin{equation}\label{cpx 2}
\begin{split}
&\nabla_h\star_\xi\alpha_1+[\Phi,\star_\xi\alpha_2]=0,\\
&\nabla_h\alpha_1-[\Phi,\alpha_2]=0,\\
&\nabla_h\alpha_2+[\alpha_1,\Phi]=0,\\
&\nabla_h\star_\xi \alpha_2+[\alpha_1,\star_\xi\Phi]=0.
\end{split}
\end{equation}
Then it is easy to check that if $\alpha\in\mathrm{Ker}D^\ast_1\cap\mathrm{Ker}D_2$, then  $I\alpha$, $J\alpha$, and $K\alpha$  satisfies (\ref{cpx 2}) and hence $I\alpha,J\alpha,K\alpha\in\mathrm{Ker}D^\ast_1\cap\mathrm{Ker}D_2$. Hence $(\mathbb{H}^1,I,J,K)$ is a quaternion vector space. These $I,J,K$ induce almost complex structures to $\mathcal{M}^{\mathrm{irr}}_{\mathrm{BaHit}}$ and we denote as $\mathcal{I},\mathcal{J},\mathcal{K}$ for the corresponding almost complex structures. It is clear that $\mathcal{I},\mathcal{J},\mathcal{K}$ satisfies the quaternion relationship.\par
To compatibility of $g$ with $\mathcal{I},\mathcal{J},\mathcal{K}$ can be shown by using the following equality: Let $A, B\in A^1_{B}(\mathfrak{u}(E))$. 
Then we have 
\begin{align*}
\mathrm{Tr}(A\wedge\star_\xi B)&=\mathrm{Tr}(A^{1,0}\wedge\star_\xi B^{0,1})+\mathrm{Tr}(A^{0,1}\wedge\star_\xi B^{1,0})\\
&=\sqrt{-1}\mathrm{Tr}(A^{1,0}\wedge B^{0,1})-\sqrt{-1}\mathrm{Tr}(A^{0,1}\wedge B^{1,0})\\
&=-\mathrm{Tr}(\star_\xi A^{1,0}\wedge B^{0,1})-\mathrm{Tr}(\star_\xi A^{0,1}\wedge B^{1,0})\\
&=-\mathrm{Tr}(\star_\xi A\wedge B).
\end{align*}
We now show $(\mathcal{M}^{\mathrm{irr}}_{\mathrm{BaHit}},g,\mathcal{I},\mathcal{J},\mathcal{K})$ is a hyperK\"ahler manifold.
 Let $\omega_{\mathcal{I}},\omega_{\mathcal{J}},\omega_{\mathcal{K}}$ be the corresponding K\"ahler forms. 
 We give the explicit form of $\omega_{\mathcal{I}},\omega_{\mathcal{J}},\omega_{\mathcal{K}}$ for $[(\nabla_h,\Phi)]\in \mathcal{M}^{\mathrm{irr}}_{\mathrm{BaHit}}$ and $\alpha=(\alpha_1,\alpha_2), \beta=(\beta_1,\beta_2)\in \mathbb{H}^1\simeq T_{[(\nabla_h,\Phi)]}\mathcal{M}^{\mathrm{irr}}_{\mathrm{BaHit}}$ for convenience.
\begin{align*}
\omega_{\mathcal{I},[(\nabla_h,\Phi)]}(\alpha,\beta)&=\int_M\mathrm{Tr}(\alpha_1\wedge\beta_1-\alpha_2\wedge\beta_2)\wedge\eta,\\
\omega_{\mathcal{J},[(\nabla_h,\Phi)]}(\alpha,\beta)&=\int_M\mathrm{Tr}(\alpha_1\wedge\star_\xi\beta_2-\alpha_2\wedge\star_\xi\beta_1)\wedge\eta,\\
\omega_{\mathcal{K},[(\nabla_h,\Phi)]}(\alpha,\beta)&=-\int_M\mathrm{Tr}(\alpha_1\wedge\beta_2+\alpha_2\wedge\beta_1)\wedge\eta.
\end{align*}
\begin{proposition}\label{der-k}
The K\"ahler form $\omega_{\mathcal{I}}$ on $\mathcal{M}^{\mathrm{irr}}_{\mathrm{BaHit}}$ satisfies at $\alpha=0$ in a slice neighborhood $\mathcal{M}^{\mathrm{irr}}_{\mathrm{BaHit}}\cap S_{(\nabla_h,\Phi),\epsilon}$
\begin{equation*}
X\omega_{\mathcal{I},[(\nabla_h,\Phi)]}(Y,Z)=0
\end{equation*}
for every $X,Y,Z\in T_0 (\mathcal{M}^{\mathrm{irr}}_{\mathrm{BaHit}}\cap S_{(\nabla_h,\Phi),\epsilon})=\mathbb{H}^1.$
\end{proposition}
\begin{proof}
We give the proof by direct computation.
\begin{equation*}
\begin{split}
X\omega_{\mathcal{I},[(\nabla_h,\Phi)]}(Y,Z)=&\frac{d}{dt}\omega_{\mathcal{I},[(\nabla_h,\Phi)+c(t)]}(\overline Y_{c(t)}, \overline Z_{c(t)})\bigg|_{t=0}\\
=&\frac{d}{dt}g_{[(\nabla_h,\Phi)+c(t)]}(\overline Y_{c(t)},\mathcal{I} \overline Z_{c(t)})\bigg|_{t=0}\\
=&\frac{d}{dt}\int_M\mathrm{Tr}\bigg((H_{(\nabla_h,\Phi)+c(t)}\overline Y_{c(t)})_1\wedge(H_{(\nabla_h,\Phi)+c(t)}\overline Z_{c(t)})_1\bigg)\wedge\eta\bigg|_{t=0}\\
&-\frac{d}{dt}\int_M\mathrm{Tr}\bigg((H_{(\nabla_h,\Phi)+c(t)}\overline Y_{c(t)})_2\wedge(H_{(\nabla_h,\Phi)+c(t)}\overline Z_{c(t)})_2\bigg)\wedge\eta\bigg|_{t=0}\\
=&\int_M\mathrm{Tr}\bigg(\frac{d}{dt}(H_{(\nabla_h,\Phi)+c(t)}\overline Y_{c(t)})_1\bigg|_{t=0}\wedge Z_1\bigg)\wedge\eta+\int_M\mathrm{Tr}\bigg(Y_1\wedge\frac{d}{dt}(H_{(\nabla_h,\Phi)+c(t)}\overline Z_{c(t)})_1\bigg|_{t=0}\bigg)\wedge\eta\\
&-\int_M\mathrm{Tr}\bigg(\frac{d}{dt}(H_{(\nabla_h,\Phi)+c(t)}\overline Y_{c(t)})_2\bigg|_{t=0}\wedge Z_2)-\int_M\mathrm{Tr}\bigg(Y_2\wedge\frac{d}{dt}(H_{(\nabla_h,\Phi)+c(t)}\overline Z_{c(t)})_2\bigg|_{t=0}\bigg)\wedge\eta.\\
\end{split}
\end{equation*}
Here $(H_{(\nabla_h,\Phi)+c(t)}\overline Y_{c(t)})_i$ (resp. $(H_{(\nabla_h,\Phi)+c(t)}\overline Z_{c(t)})_i$) is the $i$-th componet of the $H_{(\nabla_h,\Phi)+c(t)}\overline Y_{c(t)}$ (resp. $H_{(\nabla_h,\Phi)+c(t)}\overline Z_{c(t)}$).\par
The following Claim will give us the proof of the Proposition.
\begin{claim}
\begin{equation*}
\int_M\mathrm{Tr}\bigg(\frac{d}{dt}(H_{(\nabla_h,\Phi)+c(t)}\overline Y_{c(t)})_1\bigg|_{t=0}\wedge Z_1)\wedge\eta-\int_M\mathrm{Tr}\bigg(\frac{d}{dt}(H_{(\nabla_h,\Phi)+c(t)}\overline Y_{c(t)})_2\bigg|_{t=0}\wedge Z_2)\wedge\eta=0.
\end{equation*}
\end{claim}
\begin{proof}
By Lemma \ref{der-m1} and \ref{der-m2}, we have 
\begin{align*}
&\int_M\mathrm{Tr}\bigg(\frac{d}{dt}(H_{(\nabla_h,\Phi)+c(t)}\overline Y_{c(t)})_1\bigg|_{t=0}\wedge Z_1\bigg)\wedge\eta\\
=&\int_M\mathrm{Tr}\bigg(\bigg(\bigg(\frac{d}{dt}H_{(\nabla_h,\Phi)+c(t)}\bigg|_{t=0}\bigg)Y\bigg)_1\wedge Z_1\bigg)\wedge\eta\\
=&\int_M\mathrm{Tr}\bigg(\bigg(-D_1G[X,Y]^2-D^\ast_2G\widetilde{[X,Y]}\bigg)_1\wedge Z_1\bigg)\wedge\eta\\
=&\int_M\mathrm{Tr}\bigg(\big(-\nabla_hG[X,Y]^2-\nabla^\ast_h\big(G\widetilde{[X,Y]}\big)_1-[(\star_\xi\Phi)^\ast]\big(G\widetilde{[X,Y]}\big)_2-[(\star_\xi\Phi)]\big(G\widetilde{[X,Y]}\big)_3\big)\wedge Z_1\bigg)\wedge\eta\\
=&\bigg(-\nabla_hG[X,Y]^2-\nabla^\ast_h\big(G\widetilde{[X,Y]}\big)_1-[(\star_\xi\Phi)^\ast,\big(G\widetilde{[X,Y]}\big)_2]-[\Phi^\ast,\big(G\widetilde{[X,Y]}\big)_3],\star_\xi Z_1\bigg)_{L^2}\\
=&-\bigg(\nabla_hG[X,Y]^2,\star_\xi Z_1\bigg)_{L^2}-\bigg(\nabla^\ast_h\big(G\widetilde{[X,Y]}\big)_1,\star_\xi Z_1\bigg)_{L^2}\\
&-\bigg([(\star_\xi\Phi)^\ast,\big(G\widetilde{[X,Y]}\big)_2],\star_\xi Z_1\bigg)_{L^2}-\bigg([\Phi^\ast,\big(G\widetilde{[X,Y]}\big)_3],\star_\xi Z_1\bigg)_{L^2}\\
=&-\bigg(G[X,Y]^2,\nabla^\ast_h\star_\xi Z_1\bigg)_{L^2}-\bigg(\big(G\widetilde{[X,Y]}\big)_1,\nabla_h\star_\xi Z_1\bigg)_{L^2}\\
&-\bigg(\big(G\widetilde{[X,Y]}\big)_2,[\star_\xi\Phi,\star_\xi Z_1]\bigg)_{L^2}-\bigg(\big(G\widetilde{[X,Y]}\big)_3,[\Phi,\star_\xi Z_1]\bigg)_{L^2}.
\end{align*}
Here $[X,Y]^2$ is the map we defined in Lemma \ref{der-m2}.
We also have 
\begin{align*}
&\int_M\mathrm{Tr}\bigg(\frac{d}{dt}(H_{(\nabla_h,\Phi)+c(t)}\overline Y_{c(t)})_2\bigg|_{t=0}\wedge Z_2)\wedge\eta\\
=&\int_M\mathrm{Tr}\bigg(\bigg(\bigg(\frac{d}{dt}H_{(\nabla_h,\Phi)+c(t)}\bigg|_{t=0}\bigg)Y\bigg)_2\wedge Z_2\bigg)\wedge\eta\\
=&\int_M\mathrm{Tr}\bigg(\bigg(-D_1G[X,Y]^2-D^\ast_2G\widetilde{[X,Y]}\bigg)_2\wedge Z_2\bigg)\wedge\eta\\
=&\int_M\mathrm{Tr}\bigg(\big(-[\Phi,G[X,Y]^2]+[\Phi^\ast,\big(G\widetilde{[X,Y]}\big)_1]+\star_\xi\nabla^\ast_h\big(G\widetilde{[X,Y]}\big)_2-\nabla^\ast_h\big(G\widetilde{[X,Y]}\big)_3\big)\wedge Z_2\bigg)\wedge\eta\\
=&-\bigg(G[X,Y]^2,[\Phi^\ast,\star_\xi Z_2]\bigg)_{L^2}+\bigg(\big(G\widetilde{[X,Y]}\big)_1,[\Phi,\star_\xi Z_2]\bigg)_{L^2}\\
&-\int_M\mathrm{Tr}\bigg(\nabla^\ast_h\big(G\widetilde{[X,Y]}\big)_2\wedge \star_\xi Z_2\bigg)\wedge\eta-\bigg(\big(G\widetilde{[X,Y]}\big)_3,\nabla_h\star_\xi Z_2\bigg)_{L^2}\\
=&-\bigg(G[X,Y]^2,[\Phi^\ast,\star_\xi Z_2]\bigg)_{L^2}+\bigg(\big(G\widetilde{[X,Y]}\big)_1,[\Phi,\star_\xi Z_2]\bigg)_{L^2}\\
&-\bigg(\big(G\widetilde{[X,Y]}\big)_2, \nabla_h\star_\xi\star_\xi Z_2\bigg)_{L^2}-\bigg(\big(G\widetilde{[X,Y]}\big)_3,\nabla_h\star_\xi Z_2\bigg)_{L^2}.
\end{align*}
Hence we have 
\begin{align*}
&\int_M\mathrm{Tr}\bigg(\frac{d}{dt}(H_{(\nabla_h,\Phi)+c(t)}\overline Y_{c(t)})_1\bigg|_{t=0}\wedge Z_1\bigg)\wedge\eta-\int_M\mathrm{Tr}\bigg(\frac{d}{dt}(H_{(\nabla_h,\Phi)+c(t)}\overline Y_{c(t)})_2\bigg|_{t=0}\wedge Z_2\bigg)\wedge\eta\\
=&-\bigg(G[X,Y]^2,\nabla^\ast_h\star_\xi Z_1\bigg)_{L^2}-\bigg(\big(G\widetilde{[X,Y]}\big)_1,\nabla_h\star_\xi Z_1\bigg)_{L^2}\\
&-\bigg(\big(G\widetilde{[X,Y]}\big)_2,[\star_\xi\Phi,\star_\xi Z_1]\bigg)_{L^2}-\bigg(\big(G\widetilde{[X,Y]}\big)_3,[\Phi,\star_\xi Z_1]\bigg)_{L^2}\\
&+\bigg(G[X,Y]^2,[\Phi^\ast,\star_\xi Z_2]\bigg)_{L^2}-\bigg(\big(G\widetilde{[X,Y]}\big)_1,[\Phi,\star_\xi Z_2]\bigg)_{L^2}\\
&-\bigg(\big(G\widetilde{[X,Y]}\big)_2, \nabla_h \star_\xi\star_\xi Z_2\bigg)_{L^2}+\bigg(\big(G\widetilde{[X,Y]}\big)_3,\nabla_h\star_\xi Z_2\bigg)_{L^2}\\
=&-\bigg(\big(G\widetilde{[X,Y]}\big)_2,D^\ast_1IZ\bigg)_{L^2}-\bigg(\big(G\widetilde{[X,Y]}\big)_1,\big(D_2IZ\big)_1\bigg)_{L^2}-\bigg(\big(G\widetilde{[X,Y]}\big)_2,\big(D_2IZ\big)_3\bigg)_{L^2}-\bigg(\big(G\widetilde{[X,Y]}\big)_3,\big(D_2IZ\big)_2\bigg)_{L^2}\\
=&0.
\end{align*}
The last equation holds since $I$ preserves $\mathbb{H}^1$.
\end{proof}
The Proposition follows immediately from the Claim.
\end{proof}
Integrability of $\mathcal{I}$ follows from Proposition \ref{der-m} and \ref{der-k}: These two Propositions show that $\mathcal{I}$ is flat with respect to the Levi-Civita connection of $g$ and hence $\mathcal{I}$ is integrable. Although we only proved for $\mathcal{I}$, we are able to show the integrability of $\mathcal{J}$ and $\mathcal{K}$ in the same way as $\mathcal{I}$. Hence we omit the proof. From the discussion above, we have 
\begin{theorem}\label{main-th}
$(\mathcal{M}^{\mathrm{irr}}_{\mathrm{BaHit}},g,\mathcal{I},\mathcal{J},\mathcal{K})$ is a smooth hyperK\"ahler manifold.
\end{theorem}

\begin{remark}
Although Theorem \ref{main-th} may be viewed from the perspective of infinite-dimensional hyperK\"ahler reduction,
we present here a direct analytic construction of the hyperK\"ahler structure.
\end{remark}

 \subsection{Dimension of $\mathcal{M}^{\mathrm{irr}}_{\mathrm{BaHit}}$}\label{dim sec}
 In this section, we calculate the dimension of $\mathcal{M}^{\mathrm{irr}}_{\mathrm{BaHit}}$. 
 We calculate it under the assumption of $M$ being regular and $E$ being regular.
We first prove the following proposition.
 
  \begin{proposition}\label{bahit to bahiggs}
 Let $E$ be a basic bundle over $M$ with a basic metric $h$. 
 Let $(\nabla_h,\Phi)\in\mathcal{A}_{\mathrm{BaHit}}$ and let $(\overline\partial_E,\theta)$ be the associated basic Higgs bundle (See Section \ref{Hi-BH}). 
 Then the map
 \begin{equation*}
\begin{array}{rccc}
f\colon &A ^1_B(\mathfrak{u}(E))^{\oplus 2}                    &\longrightarrow& A^1_B(\mathrm{End}E)                     \\
        & \rotatebox{90}{$\in$}&               & \rotatebox{90}{$\in$} \\
        & (\alpha_1,\alpha_2)                   & \longmapsto   & \alpha_1+\sqrt{-1}\alpha_2
\end{array}
 \end{equation*}
 induces an isomorphism
 \begin{equation*}
 f:\mathbb{H}^1\to \mathbb{H}^1_{\mathrm{BaDol}}.
 \end{equation*}
 Here $\mathbb{H}^1$ is the first cohomology of the complex (\ref{cpx}) and  $\mathbb{H}^1_{\mathrm{BaDol}}$ is the first cohomology of the following complex:
 \begin{equation*}
0\longrightarrow A_B(\mathrm{End}E)\overset{\overline\partial_E+\theta}{\longrightarrow} A^1_B(\mathrm{End}E)\overset{\overline\partial_E+\theta}{\longrightarrow} A^2_B(\mathrm{End}E)\longrightarrow 0.
 \end{equation*}
 \end{proposition}
 \begin{proof}
It is enough to show that $f$ induces an isomorphism 
\begin{equation*}
f:\mathrm{Ker}D^\ast_1\cap\mathrm{Ker}D_2\to\mathrm{Ker}(\overline\partial_E+\theta)^\ast\cap\mathrm{Ker}(\overline\partial_E+\theta).
\end{equation*}
Here $(\overline\partial_E+\theta)^\ast$ is the $L^2$ adjoint of $\overline\partial_E+\theta$.\par
Let $(\alpha_1,\alpha_2)\in A ^1_B(\mathfrak{u}(E))^{\oplus 2}$. We assume that $(\alpha_1,\alpha_2)\in\mathrm{Ker}D^\ast_1\cap\mathrm{Ker}D_2$. We first show that $f(\alpha_1,\alpha_2)=\alpha_1+\sqrt{-1}\alpha_2\in\mathrm{Ker}(\overline\partial_E+\theta)^\ast\cap\mathrm{Ker}(\overline\partial_E+\theta)$. By (\ref{cpx 1}) and (\ref{dual D1}), we have
\begin{align}
\nabla_h\alpha_1-[\Phi,\alpha_2]&=0,\label{eq1}\\
\nabla_h\alpha_2+[\alpha_1,\Phi]&=0,\label{eq2}\\
\nabla_h\star_\xi \alpha_2+[\alpha_1,\star_\xi\Phi]&=0,\label{eq3}\\
\nabla_h\star_\xi\alpha_1+[\Phi,\star_\xi\alpha_2]&=0\label{eq4}.
\end{align}
Note that from Lemma \ref{bra}, (\ref{eq3}) is equivalent to 
\begin{equation}\label{eq3-1}
\nabla_h\star_\xi \alpha_2-[\star_\xi\alpha_1,\Phi]=0.
\end{equation}
Since $\star_\xi|_{A^{1,0}_B(M)}=-\sqrt{-1}\mathrm{Id}_{A^{1,0}_B(M)}$ and $\star_\xi|_{A^{0,1}_B(M)}=\sqrt{-1}\mathrm{Id}_{A^{0,1}_B(M)}$, by calculating $(\ref{eq1})+\sqrt{-1}(\ref{eq4})$ and $(\ref{eq2})+\sqrt{-1}(\ref{eq3-1})$ we have 
\begin{align*}
\nabla^{0,1}_h\alpha^{1,0}_1-[\Phi^{1,0},\alpha^{0,1}_2]&=0,\\
\nabla^{0,1}_{h}\alpha^{1,0}_2+[\alpha^{0,1}_1,\Phi^{1,0}]&=0.
\end{align*}
Since $\overline\partial_E=\nabla^{0,1}_h$ and $\theta=\sqrt{-1}\Phi^{1,0}$ (See Section \ref{Hi-BH}), they show $\alpha_1+\sqrt{-1}\alpha_2\in \mathrm{Ker}(\overline\partial_E+\theta)$. By using a similar argument, we can also show that $\alpha_1+\sqrt{-1}\alpha_2\in \mathrm{Ker}(\overline\partial_E+\theta)^\ast$. Hence $f(\mathrm{Ker}D^\ast_1\cap\mathrm{Ker}D_2)\subset\mathrm{Ker}(\overline\partial_E+\theta)^\ast\cap\mathrm{Ker}(\overline\partial_E+\theta)$ holds. We now construct the inverse of $f$ and prove the claim. Let 
 \begin{equation*}
\begin{array}{rccc}
g\colon & A^1_B(\mathrm{End}E)                   &\longrightarrow& A ^1_B(\mathfrak{u}(E))^{\oplus 2}                     \\
        & \rotatebox{90}{$\in$}&               & \rotatebox{90}{$\in$} \\
        & A                 & \longmapsto   & (\frac{A-A^\dagger_h}{2},-\sqrt{-1}\frac{A+A^\dagger_{h}}{2})
\end{array}.
 \end{equation*}
Here $A^\dagger_h$ is the formal adjoint of $A$ with respect to $h$. 
It is straightforward to check that $f\cdot g=\mathrm{Id}$ and $g\cdot f=\mathrm{Id}$ holds. Hence, it is enough to show that  $g(\mathrm{Ker}(\overline\partial_E+\theta)^\ast\cap\mathrm{Ker}(\overline\partial_E+\theta))\subset\mathrm{Ker}D^\ast_1\cap\mathrm{Ker}D_2$ to prove the claim.\par
Let $A\in \mathrm{Ker}(\overline\partial_E+\theta)^\ast\cap\mathrm{Ker}(\overline\partial_E+\theta)$. 
From \cite{BH}, we have $(\overline\partial_E+\theta)^\ast=(\nabla^{0,1}_h+\sqrt{-1}\Phi^{1,0})^\ast=\sqrt{-1}[\Lambda,\nabla^{1,0}_h+\sqrt{-1}\Phi^{0,1}]$. Since $A_B^{2,0}(M)=A^{0,2}_B(M)=0$, $A$ satisfies the following equations
\begin{align}
\nabla_h^{0,1}A^{1,0}+\sqrt{-1}[\Phi^{1,0},A^{0,1}]&=0,\label{eq5}\\
\Lambda(\nabla_h^{1,0}A^{0,1}+\sqrt{-1}[\Phi^{0,1},A^{1,0}])&=0.
\end{align}
We wedge $d\eta\mathrm{Id}_E$ to the second equation and we obtain
\begin{equation}\label{eq6}
\nabla_h^{1,0}A^{0,1}+\sqrt{-1}[\Phi^{0,1},A^{1,0}]=0.
\end{equation}
We take the formal adjoint of (\ref{eq5}) and (\ref{eq6}) with respect to $h$ and obtain
\begin{align}
\nabla_h^{1,0}(A^{1,0})^\dagger_h-\sqrt{-1}[\Phi^{0,1},(A^{0,1})^\dagger_h]&=0,\label{eq7}\\
\nabla_h^{0,1}(A^{0,1})^\dagger_h-\sqrt{-1}[\Phi^{1,0},(A^{1,0})^\dagger_h]&=0 \label{eq8}.
\end{align}
We now prove that $g(A)\in \mathrm{Ker}D^\ast_1\cap\mathrm{Ker}D_2$. 
We only prove that $g(A)=(\frac{A-A^\dagger_h}{2},-\sqrt{-1}\frac{A+A^\dagger_{h}}{2})$ satisfies (\ref{eq1}). The other can be proved by using the same argument below.
\begin{align*}
&\nabla_h\bigg(\frac{A-A^\dagger_h}{2}\bigg)+\sqrt{-1}\bigg[\Phi,\frac{A+A^\dagger_{h}}{2}\bigg]\\
=&\nabla^{1,0}_h\bigg(\frac{A^{0,1}-(A^{1,0})^\dagger_h}{2}\bigg)+\nabla^{0,1}_h\bigg(\frac{A^{1,0}-(A^{0,1})^\dagger_h}{2}\bigg)+\sqrt{-1}\bigg[\Phi^{1,0},\frac{A^{0,1}+(A^{1,0})^\dagger_{h}}{2}\bigg]+\sqrt{-1}\bigg[\Phi^{0,1},\frac{A^{1,0}+(A^{0,1})^\dagger_{h}}{2}\bigg]\\
=&0.
\end{align*}
The last equation follows from (\ref{eq5}), (\ref{eq6}), (\ref{eq7}) and (\ref{eq8}).
  \end{proof}
From now on, we assume that $(M,(T^{1,0},S,I),(\eta,\xi))$ and $E$ are regular.
See Section \ref{Sasa} for the definition of regularity of $M$ and Section \ref{reg bun} for the definition of regularity of $E$.\par
Let $(E,\overline\partial_E,\theta)$ be a basic Higgs bundle.
 We say that it is a regular basic Higgs bundle if $E$ is regular. 
 We recall that there is a one-on-one correspondence between a regular basic Higgs bundle on $M$ and a Higgs bundle over $M/S^1$ following \cite{BH2}. 
 Note that since $M$ is regular, $M/S^1$ is a Riemann surface.
  From now on, we assume the genus of $M/S^1$ is at least 2.\par

We first review the construction of a regular basic Higgs bundle over $M$ from a Higgs bundle over $M/S^1$. Let $(\widetilde{E},\overline\partial_{\widetilde{E}},\widetilde{\theta})$ be a Higgs bundle over $M/S^1$. Let $\{U_\alpha\}_{\alpha\in A}$ be a open covering of $M/S^1$. We assume that $\widetilde{E}$ is trivialized over each $U_\alpha$. Then we have a family of holomorphic transition function $\widetilde{g}_{\alpha\beta}:U_\alpha\cap U_\beta\to GL(r,\mathbb{C})$ such that it satisfies the 1-cocycle condition. Since $M$ is the total space of a $S^1$-bundle over $M/S^1$, we can regard $\{U_\alpha\times S^1\}_{\alpha\in A}$ as an open covering of $M$. We define a family of maps $g_{\alpha\beta}:U_\alpha\times S^1\cap U_\beta\times S^1\to GL(r,\mathbb{C})$ as $g_{\alpha\beta}(x,t):=\widetilde{g}_{\alpha\beta}(x)$. This family defines a vector bundle $E$ over $M$ since it satisfies the 1-cocycle condition. Since $E$ is trivialized over each $U_\alpha\times S^1$, $E$ is regular and since the transition function is constant along the $S^1$-action, $E$ is basic and finally, since $\widetilde{g}_{\alpha\beta}$ is holomorphic, $E$ is basic holomorphic. We can also show that $\overline\partial_{\widetilde{E}}$ induces a basic holomorphic structure $\overline\partial_{E}$ as follows: We assume that $\overline\partial_{\widetilde{E}}|_{U_\alpha}=\overline\partial+\widetilde{A}_\alpha$ where $\widetilde{A}_\alpha\in A^{0,1}(\mathfrak{gl}(r,\mathbb{C}))$. Then $\widetilde{A}_\beta=\widetilde{g}_{\alpha\beta}^{-1}\widetilde{A}_\alpha \widetilde{g}_{\alpha\beta}+\widetilde{g}_{\alpha\beta}^{-1}\overline\partial \widetilde{g}_{\alpha\beta}$ holds. We define $A_\alpha\in A^{0,1}_{B}(\mathfrak{gl}(r,\mathbb{C}))$ as $A_\alpha(x,t):=\widetilde{A}_\alpha(x)$. This satisfies $A_\beta=g_{\alpha\beta}^{-1}A_\alpha g_{\alpha\beta}+g_{\alpha\beta}^{-1}\overline\partial_\xi g_{\alpha\beta}$. Here $\overline\partial_\xi$ is the $(0,1)$-part of $d|_{A^{\bullet}_B(M)}$. Hence $\{A_{\alpha}\}_{\alpha\in A}$ defines a $(0,1)$-differential operator $\overline\partial_E$ and hence a basic holomorphic structure on $E$. We can also show that $\widetilde{\theta}$ induces a basic Higgs field $\theta$ by using a similar argument.\par

We next review the converse construction. Let $(E,\overline\partial_E,\theta)$ be a regular basic Higgs bundle over $M$. Since $M$ is the total space of a $S^1$-bundle over $M/S^1$ there exist an open cover of $\{ U_\alpha\}_{\alpha\in A}$ of $M/S^1$ such that $\{ U_\alpha\times S^1\}_{\alpha\in A}$ is an open cover of $M$. Since $E$ is regular, we may assume $E$ is trivialized over each $\{ U_\alpha\times S^1\}_{\alpha\in A}$ after shrinking $U_\alpha$ appropriately. Since $E$ is basic, the transition function $g_{\alpha\beta}:U_\alpha\times S^1\cap U_\beta\times S^1\to GL(r,\mathbb{C})$ of $E$ is constant along $S^1$. Hence $g_{\alpha\beta}$ reduces to the function on $U_\alpha\cap U_\beta$ and defines a vector bundle $\widetilde{E}$ on $M/S^1$. We can use a similar argument above to show that $\overline\partial _E$ and $\theta$ reduces to $\widetilde{E}$ and 
define a holomorphic structure $\overline\partial_{\widetilde{E}}$ and a Higgs field $\widetilde{\theta}$ on $\widetilde{E}$. \par
We now assume $E$ is regular. We show that there exists a one-on-one correspondence between the space of basic sections $A_B(E)$ over $M$ and smooth sections $A(\widetilde{E})$ over $M/S^1$.\par
Let $s\in A_B(E)$. Then $s_\alpha:=s|_{U_\alpha}=(s_{\alpha,1},\dots,s_{\alpha,r}):U_\alpha\times S^1\to \mathbb{C}^r$ is a basic function. Hence $s_\alpha$ reduces to a function $\widetilde{s}_\alpha:U_\alpha\to \mathbb{C}^r$. We can glue $\{\widetilde{s}_\alpha\}_{\alpha\in A}$ and define a smooth section $s$ over $\widetilde{E}$. Conversely, let $\widetilde{s}\in A(\widetilde{E})$. Then $\widetilde{s}_\alpha:= \widetilde{s}|_{U_\alpha}=(\widetilde{s}_{\alpha,1},\dots,\widetilde{s}_{\alpha,r}):U_\alpha \to \mathbb{C}^r$ is a smooth function. We define $s_{\alpha}:U_\alpha\times S^1 \to \mathbb{C}^r$ as $s_{\alpha}(x,t):=\widetilde{s}_\alpha(x)$. We can glue $\{s_\alpha\}_{\alpha\in A}$ and define a smooth section $s$ over $E$. Since $s_\alpha$ is constant along $S^1$, $s$ is basic. We define linear maps
\begin{align*}
&p:A_B(E)\to A(\widetilde{E}),\\
&q:A(\widetilde{E})\to A_B(E)
\end{align*}
as $p(s):=\widetilde{s}$ and $q(\widetilde{s}):=s$. 
$p\circ q=q\circ p=Id$ is clear from the construction.\par

\begin{proposition}\label{bahiggs to higgs}
Let $(E,\overline\partial_E,\theta)$ be a regular basic Higgs bundle over $M$ and $(\widetilde{E},\overline\partial_{\widetilde{E}},\widetilde{\theta})$ be the induced Higgs bundle over $M/S^1$. Then p,q induces a morphism between complexes
\begin{align*}
&p:(A^\bullet_B(E),\overline\partial_E+\theta)\to (A(\widetilde{E}),\overline\partial_{\widetilde{E}}+\widetilde{\theta}),\\
&q:(A(\widetilde{E}), \overline\partial_{\widetilde{E}}+\widetilde{\theta} )\to (A^\bullet_B(E),\overline\partial_E+\theta).
\end{align*}
Since $p\circ q=q\circ p=Id$, $p$ and $q$ induce an isomorphism between the cohomologies. In particular, the dimensions of the cohomologies of the two complexes are the same.
\end{proposition}
\begin{proof}
This is clear from the construction of $p,q$ and $(\widetilde{E},\overline\partial_{\widetilde{E}},\widetilde{\theta})$.
\end{proof}
Let $h$ be a basic Hermitian metric and let $(\nabla_h,\Phi)\in\mathcal{A}^{\mathrm{irr}}_{\mathrm{BaHit}}$. 
Then $(E,\overline\partial_E:=\nabla^{0,1}_h, \theta:=\sqrt{-1}\Phi^{1,0})$ is a regular stable Higgs bundle (See Section \ref{Hi-BH}). 
Since $h$ is basic and $E$ is regular, we can show that $h$ induces a metric $\widetilde{h}$ on $\widetilde{E}$ by using the trivialization above. 
It is clear from the construction that $\widetilde{h}$ is a harmonic metric for  $(\widetilde{E},\overline\partial_{\widetilde{E}},\widetilde{\theta})$. 
Hence $(\widetilde{E},\overline\partial_{\widetilde{E}},\widetilde{\theta})$  is polystable and degree 0. Assume that $(\widetilde{E},\overline\partial_{\widetilde{E}},\widetilde{\theta})$ is not stable. Then by \cite[Proposition 3.3]{S1}, there exists a sub Higgs bundle $\widetilde{V}\subset\widetilde{E}$ such that  $(\widetilde{E},\overline\partial_{\widetilde{E}},\widetilde{\theta})=(\widetilde{V},\overline\partial_{\widetilde{V}},\widetilde{\theta}_V)\oplus(\widetilde{V}^\perp,\overline\partial_{\widetilde{V}^\perp},\widetilde{\theta}_{\widetilde{V}^\perp})$ holds and both Higgs bundles are stable and degree 0. 
Here $\widetilde{V}^\perp$ is the orthogonal bundle of $\widetilde{V}$. 
By applying the above procedure to $\widetilde{V}$, we obtain a sub-Higgs bundle $V\subset E$. 
The harmonic metric $\widetilde{h}|_{\widetilde{V}}$ of $\widetilde{V}$ induces a harmonic metric $h_V$ on $V$. Hence $(V,\overline\partial_V,\theta_V)$ is  degree 0. 
This contradicts to the stability of $(E,\overline\partial_E, \theta)$ and hence $(\widetilde{E},\overline\partial_{\widetilde{E}},\widetilde{\theta})$ is stable.\par
Let $\mathbb{H}_{\mathrm{Dol}}^1$ be the first cohomology of the folllowing complex
\begin{equation*}
0\longrightarrow A(\mathrm{End}\widetilde{E})\overset{\overline\partial_{\mathrm{End}\widetilde{E}}+\widetilde{\theta}}{\longrightarrow} A^1(\mathrm{End}\widetilde{E})\overset{\overline\partial_{\mathrm{End}\widetilde{E}}+\widetilde{\theta}}{\longrightarrow} A^2(\mathrm{End}\widetilde{E})\longrightarrow 0.
\end{equation*}
 Since $E$ is regular, the dual $E^\vee$ is also regular, and so is $\mathrm{End}E$. 
 We can apply Proposition \ref{bahiggs to higgs} to $\mathrm{End}E$ and obtain $\mathbb{H}_{\mathrm{BaDol}}^1\simeq \mathbb{H}_{\mathrm{Dol}}^1$.  
 By \cite{N}, $\mathrm{dim}_{\mathbb{C}}\mathbb{H}^1_{\mathrm{Dol}}=2(\mathrm{rk}\widetilde{E})^2(g-1)+2$. 
 Then combining Proposition \ref{bahit to bahiggs} and \ref{bahiggs to higgs}, we obtain 
\begin{theorem}\label{dim}
Let $(M,(T^{1,0},S,I),(\eta,\xi))$ be a regular Sasakian threefold. Let $E$ be a regular basic bundle and $h$ be a basic Hermitian metric. Let $g$ be the genus of $M/S^1$.
Then 
\begin{equation*}
 \mathrm{dim}_{\mathbb{R}}\mathcal{M}^{\mathrm{irr}}_{\mathrm{BaHit}}=4(\mathrm{rk}E)^2(g-1)+4.
 \end{equation*}
\end{theorem}
 \section{Appendix}\label{sec 5}
 \subsection{Degree of Basic Bundles}
 Let $(M,(T^{1,0},S,I),(\eta,\xi))$ be a compact Sasakian manifold of dimension $2n+1$.
 Let $E$ be a basic bundle and $D$ be a basic connection. 
 Let $F_D$ be the curvature of $D$.
  Since $E$ and $D$ are basic, $F_D\in A^2_B(\mathrm{End}E).$ 
  For any $0 \leqslant i \leqslant n$, we define $c_{i,B}(E,D)\in A^{2i}_B(M)$ by
   \begin{equation*}
   \mathrm{det}\bigg(\mathrm{Id}_E-\frac{F_D}{2\pi\sqrt{-1}}\bigg)=1+\sum^{2n}_{i=1} c_{i,B}(E,D).
      \end{equation*}
   Then, as the case of the usual Chern-Weil theory, the cohomology class,
   \begin{equation*}
   c_{i,B}(E)\in H^{2i}_B(M)
      \end{equation*}
    of each $c_{i,B}(E,D)$ is independent of the choice of a basic connection $D$.\par
    We define the \textit{degree} of $E$ as
   \begin{equation*}
   \mathrm{deg}(E):=\frac{1}{2\pi\sqrt{-1}}\int_M\mathrm{Tr}(\Lambda F_D)(d\eta)^n\wedge \eta.
   \end{equation*}
   We also have 
 \begin{equation*}
  \mathrm{deg}(E)=\int_Mc_{1,B}(M)\wedge (d\eta)^{n-1}\wedge\eta.
   \end{equation*}
 Hence $\mathrm{deg}(E)$ only depends on $E$.
 We define the slope of $E$ as 
 \begin{equation*}
 \mu(E):=\frac{\mathrm{deg} (E)}{\mathrm{rank} E}.
 \end{equation*}
 \subsection{Basic Higgs Bundles}
  Throughout this section, let $(M,(T^{1,0},S,I),(\eta,\xi))$ be a compact Sasakian manifold.\par
  Let $E$ be a basic vector bundle over $M$. We say that $E$ is transverse holomorphic if there exists a local trivialization $\{U_\alpha\}_{\alpha\in A}$ of $E$ such that the associated transition function
 $g_{\alpha\beta}:U_\alpha\cap U_\beta\to GL_r(\mathbb{C})$ is basic and holomorphic (i.e. $i_\xi dg_{\alpha\beta}=0$ and $\overline\partial_\xi g_{\alpha\beta}=0$).
    For a transversely holomorphic vector bundle $E$ over $M$, we define the Dolbeault operator 
    \begin{equation*}
    \overline\partial_E:A_B(E)\to A^{0,1}_B(E)
    \end{equation*}
    \begin{equation*}
    \overline\partial_E|_{U_\alpha}:=\overline\partial_\xi.
   \end{equation*}
    This is well defined since the transition function is holomorphic and satisfies $ \overline\partial_E \overline\partial_E=0$. It is canonically extended to $\overline\partial_E:A^{p,q}_B(E)\to A^{p,q+1}_B(E)$ and satisfies the Leibniz rule:
    \begin{equation*}
    \overline\partial_E(\omega\wedge s)=\overline\partial_\xi\omega\wedge s+(-1)^{p+q}\omega\wedge\overline\partial_Es.
        \end{equation*}
        Conversely, if we have an operator $\overline\partial_E:A^{p,q}_B(E)\to A^{p,q+1}_B(E)$ such that it satisfies $\overline\partial_E \overline\partial_E=0$ and the Leibniz rule, $\overline\partial_E$ defines a transverse holomorphic structure by the Frobenius theorem (\cite{Ko}).
    \begin{definition}
   Let $(M,(T^{1,0},S,I),(\eta,\xi))$ be a compact Sasakian manifold. 
   A basic Higgs bundle $(E,\overline\partial_E,\theta)$ over $X$ is a pair such that 
  \begin{itemize}
  \item $E$ is basic and $(E, \overline\partial_E)$ is a transverse holomorphic bundle.
  \item $\theta\in A^{1,0}_B(\mathrm{End}E)$, $\overline\partial_E\theta=0$, and $\theta\wedge\theta=0$.
  \end{itemize}
  We call $\theta$ a Higgs field.
 \end{definition}
 Let $(E,\overline\partial_E,\theta)$ be a basic Higgs bundle on $M$ and $h$ be a basic hermitian metric.\par
  We define a connection $\nabla_h:A(E)\to A^1(E)$ as follows: Let $e_{1,\alpha},\dots, e_{r,\alpha}$ be a local holomorphic frame of $E$ on $U_\alpha$ and $H_\alpha:=(h(e_{i,\alpha},e_{j,\alpha})_{1 \leqslant i,j  \leqslant r})$. 
 We define 
 \begin{equation*}
 \nabla_h|_{U_\alpha}:=d+H^{-1}_\alpha\partial_\xi H_{\alpha}.
 \end{equation*}
 This is well defined, and since $h$ is basic, $\nabla_h$ is a basic connection. 
 $\nabla_h$ is also a $h$-unitary connection. Note that $\nabla^{0,1}_h=\overline\partial_E.$\par
 Let $\theta^\dagger_h$ be the formal adjoint of $\theta$: 
  For every section $u,v\in A(E)$,
 \begin{equation*}
 h(\theta u,v)=h(u, \theta^\dagger_h v)
 \end{equation*}
 holds.
 We define a connection $D_h:=\nabla_h+\theta+\theta^\dagger_h$. This is a basic connection. Let $F_{D_h}$ be the curvature of $D_h$. We say that $h$ is \textit{Hermitie-Einstein} if 
 \begin{equation*}
 \Lambda F_{D_h}^\perp=0.
 \end{equation*}
 Here $\Lambda F_{D_h}^\perp$ is the trace-free part of $F_{D_h}$.  
 \par
 The existence of Hermitie-Einstein metric is related to the stability of the Higgs bundle.
 We now recall the them following \cite{BHa, BS}. \par
 Let  $(E,\overline\partial_E,\theta)$ be a basic Higgs bundle on $M$.
  Let $\mathcal{O}_{B}$ be the sheaf of basic holomorphic functions and $\mathcal{O}_B(E)$ be the sheaf of basic holomorphic sections of $E$. A \textit{sub Higgs sheaf} of $(E,\overline\partial_E,\theta)$ is a coherent $\mathcal{O}_{B}$-subsheaf $\mathcal{V}$ of  $\mathcal{O}_B(E)$ such that $\theta(\mathcal{V})\subset \mathcal{V}\otimes\Omega^1_B$. 
 Here $\Omega^1_B$ is the sheaf of basic holomorphic 1-form. By \cite{BHa}, if $\mathrm{rk}\mathcal{V}<\mathrm{rk}E$ and $\mathcal{O}_B(E)/\mathcal{V}$ is torsion-free, then there is a transversely analytic sub-variety $S\subset M$ of complex co-dimension at least 2 such that $\mathcal{V}|_{M\backslash S}$ is a transverse holomorphic bundle on $M\backslash S$. 
 We define the degree of $\mathcal{V}$ as the degree of $\mathcal{V}|_{M\backslash S}.$
 \begin{definition}
 A basic Higgs bundle $(E,\overline\partial_E,\theta)$ is \textit{stable} if 
 \begin{itemize}
 \item $E$ admits a basic Hermitian metric $h$.
 \item For every sub-Higgs sheaf $\mathcal{V}\subset \mathcal{O}_B(E)$ such that $\mathrm{rk}\mathcal{V}<\mathrm{rk}E$ and $\mathcal{O}_B(E)/\mathcal{V}$ is torsion-free,
 \begin{equation*}
 \frac{\mathrm{deg}(\mathcal{V})}{\mathrm{rk}\mathcal{V}}<\frac{\mathrm{deg}(E)}{\mathrm{rk}E}.
  \end{equation*}
  holds.
  \end{itemize}
  We say that $(E,\overline\partial_E,\theta)$ is \textit{polystable} if 
  \begin{equation*}
 (E,\overline\partial_E,\theta)=\bigoplus_i(E_i,\overline\partial_{E_i},\theta_i)
  \end{equation*}
  where each $(E_i,\overline\partial_{E_i},\theta_i)$ is stable and 
  \begin{equation*}
   \frac{\mathrm{deg}(E)}{\mathrm{rk}E}=\frac{\mathrm{deg}(E_i)}{\mathrm{rk}E_i}.
     \end{equation*}
 \end{definition}
 The following theorem is the Kobayashi-Hitchin correspondence on Sasakian manifolds.
 The K\"ahler case was proved in \cite{S1, US}.
 \begin{proposition}[{\cite[Theorem 5.2, Proposition 5.3.]{BH}}]\label{BH KH1}
 For a stable basic Higgs bundle $(E,\overline\partial_E,\theta)$ over a compact Sasakian manifold  $(M,(T^{1,0},S,I),(\eta,\xi))$, there exsit a basic Hermitian metric $h$ such that $D_h$ satisfies 
 \begin{equation*}
  \Lambda F_{D_h}^\perp=0.
 \end{equation*}
Note that $h$ is a Hermite-Einstein metric.\par
Moreover, if $c_{1,B}(E)=c_{2,B}(E)=0$, then $D_h$ is flat (i.e. $F_{D_h}=0$).
\end{proposition}
If we assume some conditions for the degree of the bundle, we have the converse.
\begin{proposition}[{\cite[Theorem 4.7.]{BHa}},{\cite[Proposition 7.1.]{BH}}]\label{BH KH2}
 Let $(E,\overline\partial_E,\theta)$ be a basic Higgs bundle over a compact Sasakian manifold  $(M,(T^{1,0},S,I),(\eta,\xi))$ with  $\mathrm{deg}(E)=0$. 
Suppose that $h$ is a basic Hermitian metric on $E$ with $  \Lambda F_{D_h}=0$. 
Then $(E,\overline\partial_E,\theta)$ is a direct sum of stable basic Higgs bundles of degree 0.
\end{proposition}
       
\subsubsection{Basic Higgs Bundles and Basic Hitchin Equation}\label{Hi-BH}
In this section, we clarify the relation between a stable basic Higgs bundle and an irreducible basic Hitchin pair.\par
Let $(\nabla_h,\Phi)\in\mathcal{A}^{\mathrm{irr}}_{\mathrm{BaHit}}$. $(E,\nabla^{0,1}_h,\sqrt{-1}\Phi^{1,0})$ is a basic Higgs bundle.
We show that this Higgs bundle is stable with degree 0. Since $\Phi\in A^1_B(\mathfrak{u}(E))$, we have
\begin{equation*}
\Phi^{0,1}=-(\Phi^{1,0})_h^\dagger.
\end{equation*}
Here $(\Phi^{1,0})_h^\dagger$ is the formal adjoint of $\Phi^{1,0}$. Since $\nabla_h$ is a $h$-unitary connection and $ \nabla^{0,1}_h\Phi^{1,0}=0$, we have 
\begin{equation*}
\nabla^{1,0}_h\Phi^{0,1}=-\nabla^{1,0}_h(\Phi^{1,0})_h^\dagger=0.
\end{equation*}
Hence $D=\nabla_h+\sqrt{-1}\Phi$ is a flat bundle and $\mathrm{deg}(E)=0$. Stability of $(E,\nabla^{0,1}_h,\sqrt{-1}\Phi^{1,0})$ follows form Proposition \ref{BH KH2} and irreducibilty of $(\nabla_h,\Phi)$.\par
Let $(E,\overline\partial_E,\theta)$ be a stable basic Higgs bundle of degree 0. 
Then by Proposition \ref{BH KH1}, there exists a basic Hermitian metric $h$ such that the connection $D=\nabla_h+\theta+\theta^\dagger_h$ is flat.
Let $\Phi:=-\sqrt{-1}(\theta+\theta^\dagger_h)$. Then $(\nabla_h,\Phi)$ is an irreducible Hitchin pair.           
           
  \subsection{Harmonic Bundles}\label{sec 5.3}
 Let $M$ be a compact Riemann manifold and $E$ be a rank $r$ complex vector bundle with a Hermitian metric $h$. Let $D$ be a flat bundle. As we mentioned in the previous section we have a decomposition
 \begin{equation*}
 D=\nabla_h+\sqrt{-1}\Phi
 \end{equation*}
  such that $\nabla_h$ is a $h$-unitary connection and $\Phi$ is an $\mathrm{End}E$-valued skew-symmetric 1-form w.r.t. $h$.\par
  We recall the following theorem due to Corlette \cite{Co}.
    \begin{theorem}[\cite{Co}]\label{Co} If a flat bundle $(E,D)$ is semi-simple, then there exists a Hermitian metric $h$ on $E$ such that 
  \begin{equation*}
  (\nabla_h)^\ast\Phi=0.
  \end{equation*}
  
  Here $(\nabla_h)^\ast$ is the formal adjoint of $\nabla_h$. 
  We call the metric $h$ a harmonic metric. If $D$ is reducible, then the harmonic metric is unique up to multiplication by a constant scalar. 
  If $h$ is a harmonic metric, we call the pair $(D,h)$ a harmonic bundle.
   \end{theorem}
  From now we assume $(M,(T^{1,0},S,I),(\eta,\xi))$ to be a compact Sasakian manifold.
  Under this assumption, harmonic metrics become basic metrics:
  \begin{proposition}[{\cite[Proposition 4.1, Theorem 4.2.]{BH}}]\label{BH FH}
  Let $(M,(T^{1,0},S,I),(\eta,\xi))$ be a compact Sasakian manifold and let $(E,D)$ be a flat bundle with a Hermitian metric $h$.
   Let $D=\nabla_h+\sqrt{-1}\Phi$ be the decomposition of (\ref{decomp}). 
   Then the following are equivalent:
  \begin{itemize}
  \item $\Phi(\xi)=0,$
  \item $h$ is a basic metric i.e. $(h\in A_{B}(E^\vee\otimes\overline E^\vee))$.
  \end{itemize}
  This condition implies that $\Phi\in A^1_B(\mathfrak{u}(E))$.\par
 Moreover, when $h$ is a harmonic metric, then the following are equivalent:
 \begin{itemize}
 \item $(\nabla_h)^\ast\Phi=0$ (i.e. $h$ is a harmonic metric),
 \item The Hermitian metric $h$ is basic ($\iff \Phi(\xi)=0$ and hence $\Phi\in A^1_B(\mathfrak{u}(E))$ by above) and for the decomposition
 \begin{equation*}
 \sqrt{-1}\Phi=\theta^{1,0}_{h,\xi}+\theta^{0,1}_{h,\xi}
  \end{equation*}
  with $\theta^{1,0}_{h,\xi}\in A^{1,0}_B(\mathrm{End}E)$ and $\theta^{0,1}_{h,\xi}\in A^{0,1}_B(\mathrm{End}E)$,
  \begin{equation*}
  \overline\partial_{h,\xi} \overline\partial_{h,\xi}=0,\ \theta^{1,0}_{h,\xi}\wedge\theta^{1,0}_{h,\xi}=0,\  \overline\partial_{h,\xi} \theta^{1,0}_{h,\xi}=0.
          \end{equation*}
          Here $  \overline\partial_{h,\xi}$ is the $(0,1)$-part of $\nabla_h$. We note that $(E, \overline\partial_{h,\xi},\theta^{1,0}_{h,\xi})$ is a basic Higgs bundle.
 \end{itemize}
  \end{proposition}
 According to \cite[p.20]{BH}, combining Proposition \ref{BH KH1}, \ref{BH KH2}, and \ref{BH FH}, there is an one-on-one correspondence between the following objects on a compact Sasakian manifold $(M,(T^{1,0},S,I),(\eta,\xi))$:
 \begin{itemize}
 \item the semi-simple flat bundle $(E,D)$,
 \item the polystable basic Higgs bundle with $c_{1,B}(E)=c_{2,B}(E)=0$.
  \end{itemize}
 We also note that harmonic metrics on Sasakian manifolds have been studied recently in \cite{WZ}. 
 \subsubsection{Harmonic Bundles and Basic Hitchin Equation}\label{Ha-BH}
  In this section,  we clarify the relation between a harmonic bundle and a basic Hitchin pair.\par
   Let $(\nabla_h,\Phi)\in \mathcal{A}^{\mathrm{irr}}_{\mathrm{BaHit}}$. Since $(E,\nabla^{0,1}_h,\sqrt{-1}\Phi^{1,0})$ is a basic Higgs bundle, $(E,D=\nabla_h+\sqrt{-1}\Phi)$ is a harmonic bundle by Proposition \ref{BH FH}. Simplicity of $(E,D,h)$ follows from the irreducibility of  $(\nabla_h,\Phi)$.\par
  Let $(E,D)$ be a simple flat bundle. From Proposition \ref{Co}, we have a harmonic metric $h$. Let $D=\nabla_h+\sqrt{-1}\Phi$ be the decomposition of (\ref{decomp}), then by Proposition \ref{BH KH2} and \ref{BH FH}, $(E,\nabla^{0,1}_h,\sqrt{-1}\Phi^{1,0})$ is a stable basic Higgs bundle of degree 0. Then $(\nabla^{0,1}_h,\sqrt{-1}\Phi^{1,0})$ is an irreducible basic Hitchin pair. 

\subsection{Basic Hitchin Equation and ASD Contact Instanton}\label{sec 5.4}
In this section, we show that a solution of the basic Hitchin equation on $\mathbb{R}^3$ can be obtained by dimensional reduction of an ASD contact instanton on $\mathbb{R}^5$.
We prove this as follows: We first prove that an ASD contact instanton on $\mathbb{R}^5$ becomes a basic connection under an appropriate gauge. 
Next, we assume that the ASD contact instanton is invariant under two directions and show that it induces a solution of the basic Hitchin equation on $\mathbb{R}^3$.
\par
Before we proceed, we recall the contact structure on $\mathbb{R}^{2n+1}$ based on \cite{Bl}.
Let $(t,u)=(t,x_1,\dots, x_n,y_1,\dots,y_n)$ be the canonical coordinate of  $\mathbb{R}^{2n+1}$.
Then, we have the standard contact structure 
\begin{equation*}
\eta:=dt-\sum_iy_idx_i
\end{equation*}
with the Reeb vector field
\begin{equation*}
\xi:=\frac{\partial}{\partial t}.
\end{equation*}
We equip $\mathbb{R}^{2n+1}$ with the Riemannian metric
\begin{equation}\label{eq 5.4.7}
g = \sum_{i=1}^n (dx_i^2 + dy_i^2) + \eta^2.
\end{equation}
Then $(\eta,\xi,g)$ defines the standard Sasakian structure on $\mathbb{R}^{2n+1}$ (see e.g. \cite[Chapter 6]{Bl}).
We note that $g$ differs from the usual Euclidean metric.
\par
From now on, the Hodge star operator $\ast$ is taken with respect to the above Sasakian metric $g$. 
Moreover, the symbols $\eta$ and $\xi$ always denote the contact form and the Reeb vector field defined above.\par
We denote the trivial rank $r$ complex vector bundle $\mathbb{R}^5 \times \mathbb{C}^r$ over $\mathbb{R}^5$ (resp. $\mathbb{R}^3 \times \mathbb{C}^r$ over $\mathbb{R}^3$) by $\mathbb{C}^r_{\mathbb{R}^5}$ (resp. $\mathbb{C}^r_{\mathbb{R}^3}$).\par
 Let $\mathfrak{u}(r)$ be the Lie algebra of the unitary group $U(r)$ and $A(\mathbb{R}^5, \mathfrak{u}(r))$ be the space of $\mathfrak{u}(r)$-valued smooth functions.
 We take $A_{x_i},A_{y_i},A_t\in A^1(\mathbb{R}^5, \mathfrak{u}(r))$ and define a connection $D$ of $\mathbb{C}^r_{\mathbb{R}^5}$ as 
 \begin{equation*}
 D=d+A=d+\sum^2_{i=1}(A_{x_i}dx_i+A_{y_i}dy_i)+A_tdt.
 \end{equation*}
 We note that $D$ is a unitary connection with respect to the standard Hermitian metric of $\mathbb{C}^r_{\mathbb{R}^5}$.
   Let $F_D=D^2$ be the curvature of $D$.
Following \cite{BHa0}, we say that $D$ is an ASD contact instanton if 
  \begin{equation*}
  \ast F_D+\eta\wedge F_D=0.
  \end{equation*}
 As explained in \cite[Section~2.1, p.~568]{BHa0}, if $D$ is an ASD contact instanton, then 
 \begin{equation}\label{eq 5.4}
 i_\xi F_D=0
 \end{equation}
holds.
We now assume $D$ is an ASD contact instanton.
Using the property (\ref{eq 5.4}), we show that there exists a global gauge transformation 
\begin{equation*}
g:\mathbb{R}^5\to U(r)
\end{equation*}
such that, after applying this gauge transformation, the connection $D$ becomes a basic connection.
We construct such a gauge transformation $g$ as the solution of the ODE
\begin{equation*}
\frac{d}{dt}g(t,u)=-A_t(t,u) g(t,u), \,\,g(0,u)=Id.
\end{equation*}
  $g$ takes value in $U(r)$ since $A_t(t,u)$ is $\mathfrak{u}(r)$-valued.
Let $A(g)=\sum^2_{i=1}(A_{x_i}(g)dx_i+A_{y_i}(g)dy_i)+A_t(g)dt$ be the connection matrix of $D$ with respect to $g$.
Then we have 
\begin{equation*}
A(g)=g^{-1}Ag+g^{-1}dg.
\end{equation*}
 To prove that $D$ is basic in the gauge $g$, we need to show 
 \begin{align}
i_\xi A(g)=A_t(g)&=0,\label{eq 5.4.1}\\
i_\xi dA(g)&=0.\label{eq 5.4.2} 
 \end{align}
 The first equation (\ref{eq 5.4.1}) holds since
  \begin{align*}
 A_t(g)&=g^{-1} A_t g+g^{-1}\frac{\partial}{\partial t}g\\
 &=g^{-1} A_t g-g^{-1} A_t g\\
 &=0.
   \end{align*}
 To prove the second equation, we use the property (\ref{eq 5.4}).
 We note that (\ref{eq 5.4}) is independent of the gauge.
 \begin{align*}
 0=i_\xi F_D&=i_\xi(dA(g)+A(g)\wedge A(g))\\
 &=i_\xi dA(g).
  \end{align*}
 For the third equation, we used $i_\xi A(g)=0$.
 
\begin{remark}
In \cite[Section 2.1]{BHa0}, the authors proved that if a connection $D$ on a principal bundle $P\to M$ over a five-dimensional contact manifold $M$ is an ASD contact instanton, then the Reeb foliation lifts to $P$ and hence $P$ admits a natural transverse structure.
With respect to this structure, $D$ is a transverse (basic) connection.\par
The discussion above may be regarded as an explicit verification of this fact in the special case of the trivial bundle over $\mathbb{R}^5$. 
\end{remark}

So far, we showed that if $D$ is an ASD contact instanton, then there exists a gauge transformation such that the gauge-transformed connection is basic. 
Therefore, without loss of generality, we may assume from now on that 
 \begin{equation*}
D=d+A=d+\sum^2_{i=1}(A_{x_i}dx_i+A_{y_i}dy_i)+A_tdt.
\end{equation*}
is both an ASD contact instanton and basic. 
Since $D$ is basic, $i_\xi A=0, i_\xi dA=0$ holds and this implies
\begin{align*}
A_t=&0,\\
\frac{\partial}{\partial t}A_{x_i}= \frac{\partial}{\partial t}A_{y_i}=&0\,\, (i=1,2).
\end{align*}
Hence, the curvature $F_D$ has the form
\begin{align*}
F_D&=F_{x_1y_1}dx_1\wedge dy_1+F_{x_1x_2}dx_1\wedge dx_2+F_{x_1y_2}dx_1\wedge dy_2\\
&+F_{y_1x_2}dy_1\wedge dx_2+F_{y_1y_2}dy_1\wedge dy_2+F_{x_2y_2}dx_2\wedge dy_2.
 \end{align*}
 Recall that the Hodge star $\ast$ is taken with respect to the Sasakian metric $g$ defined in \eqref{eq 5.4.7}.
 The frame $\{dx_1,dy_1,dx_2,dy_2,\eta\}$ is orthonormal with respect to $g$.
Hence, substituting the expression of $F_D$ into the ASD contact instanton equation, we have
 \begin{equation*}
\begin{split} 
0=&\ast F_D+\eta\wedge F_D\\
=&F_{x_1y_1}\eta\wedge dx_2\wedge dy_2-F_{x_1x_2}\eta\wedge dy_1\wedge dy_2+F_{x_1y_2}\eta\wedge dy_1\wedge dx_2\\
&+F_{y_1x_2}\eta\wedge dx_1\wedge dy_2-F_{y_1y_2}\eta\wedge dx_1\wedge dx_2+F_{x_2y_2}\eta\wedge dx_1\wedge dy_1\\
&+F_{x_1y_1}\eta\wedge dx_1\wedge dy_1+F_{x_1x_2}\eta\wedge dx_1\wedge dx_2+F_{x_1y_2}\eta\wedge dx_1\wedge dy_2\\
&+F_{y_1x_2}\eta\wedge dy_1\wedge dx_2+F_{y_1y_2}\eta\wedge dy_1\wedge dy_2+F_{x_2y_2}\eta\wedge dx_2\wedge dy_2\\
=&(F_{x_2y_2}+F_{x_1y_1})\eta\wedge dx_1\wedge dy_1+(-F_{y_1y_2}+F_{x_1x_2})\eta\wedge dx_1\wedge dx_2\\
&+(F_{y_1x_2}+F_{x_1y_2})\eta\wedge dx_1\wedge dy_2+(F_{x_1y_2}+F_{y_1x_2})\eta\wedge dy_1\wedge dx_2\\
&+(-F_{x_1x_2}+F_{y_1y_2})\eta\wedge dy_1\wedge dy_2+(F_{x_1y_1}+F_{x_2y_2})\eta\wedge dx_2\wedge dy_2.
\end{split}
 \end{equation*} 
By comparing the coefficients, we obtain
  \begin{equation}\label{eq 5.4.3}
\begin{split} 
F_{x_1y_1}+F_{x_2y_2}&=0,\\
F_{x_1x_2}-F_{y_1y_2}&=0,\\
F_{x_1y_2}+F_{y_1x_2}&=0.
\end{split}
 \end{equation} 

 We now assume that $A_{x_1}, A_{y_1}, A_{x_2}, A_{y_2}$ are constant along the $x_2$- and $y_2$-directions, or equivalently, $\partial_{x_2}A_\alpha=\partial_{y_2}A_\alpha=0,\, \alpha\in\{x_1,y_1,x_2,y_2\}$ holds.
Under this assumption, the equation (\ref{eq 5.4.3}) is equivalent to 

\begin{align}
-\frac{\partial}{\partial y_1}A_{x_1}+\frac{\partial}{\partial x_1}A_{y_1}+[A_{x_1},A_{y_1}]+[A_{x_2},A_{y_2}]&=0,\label{eq 5.4.4}\\
\frac{\partial}{\partial x_1}A_{x_2}+[A_{x_1},A_{x_2}]-\frac{\partial}{\partial y_1}A_{y_2}-[A_{y_1},A_{y_2}]&=0,\label{eq 5.4.5}\\
\frac{\partial}{\partial x_1}A_{y_2}+[A_{x_1},A_{y_2}]+\frac{\partial}{\partial y_1}A_{x_2}+[A_{y_1},A_{x_2}]&=0\label{eq 5.4.6}.
\end{align} 
 
Next, we define $\nabla$ and $\Phi$ as  
\begin{align*}
&\nabla:=d+A_{x_1}dx_1+A_{y_1}dy_1\\
&\Phi:=A_{y_2}dx_1+A_{x_2}dy_1.
\end{align*}
Since $A_{x_1}, A_{y_1}, A_{x_2}, A_{y_2}$  are constant along the $x_2$- and $y_2$-directions, they descend to $\mathbb{R}^3$.
Hence $\nabla$ defines a unitary connection on $\mathbb{C}^r_{\mathbb{R}^3}$ and $\Phi$ defines a $\mathfrak{u}(r)$-valued 1-form of $\mathbb{R}^3$.
Since $D$ is basic, $\nabla$ and $\Phi$ are also basic.
By equations (\ref{eq 5.4.4}), (\ref{eq 5.4.5}), (\ref{eq 5.4.6}), we have
\begin{align*}
F_{\nabla}-\Phi\wedge\Phi=&\bigg(-\frac{\partial}{\partial y_1}A_{x_1}+\frac{\partial}{\partial x_1}A_{y_1}+[A_{x_1},A_{y_1}]+[A_{x_2},A_{y_2}]\bigg)dx_1\wedge dy_1=0,\\
\nabla\Phi=&\bigg(-\frac{\partial}{\partial y_1}A_{y_2}-[A_{y_1},A_{y_2}]+\frac{\partial}{\partial x_1}A_{x_2}+[A_{x_1},A_{x_2}]\bigg)dx_1\wedge dy_1=0,\\
\nabla\star_\xi\Phi=&\nabla(A_{y_2}dy_1-A_{x_2}dx_1)\\
=&\bigg(\frac{\partial}{\partial x_1}A_{y_2}+[A_{x_1},A_{y_2}]+\frac{\partial}{\partial y_1}A_{x_2}+[A_{y_1},A_{x_2}]\bigg)dx_1\wedge dy_1=0.&
\end{align*}
 Hence, the pair $(\nabla,\Phi)$ is a solution of the basic Hitchin equation.
 This shows that, under the above invariance assumption, the basic Hitchin equation can be obtained by dimensional reduction of the ASD contact instanton.
 
\end{document}